\documentclass[11pt]{amsart}
\usepackage[utf8]{inputenc}
\usepackage{graphicx,amssymb,amsmath,amsthm,bm}
\usepackage{caption}
\usepackage{subcaption}
\usepackage{indentfirst}
\usepackage{cancel}
\usepackage{lipsum}
\usepackage{epstopdf,epsfig}
\usepackage{enumerate,color} 
\usepackage{multirow}

\usepackage[normalem]{ulem}

\usepackage{array,ragged2e}
\newcolumntype{C}{>{\Centering\arraybackslash}m{0.14\linewidth}}

\usepackage{accents}

\usepackage[overload]{empheq}
\usepackage{soul}
\usepackage[multiple]{footmisc}
\newcommand{\bp}{\backprime}

\newcommand{\veps}{\varepsilon}

\numberwithin{equation}{section}

\theoremstyle{plain}
\newtheorem{theorem}{Theorem}[section]
\newtheorem{lemma}[theorem]{Lemma}

\theoremstyle{definition}

\newtheorem*{defi*}{Definition} 
\theoremstyle{remark}
\newtheorem{remark}{Remark}

\theoremstyle{remark}

\usepackage{etoolbox}
\makeatletter
\let\alignts@preamble\align@preamble
\patchcmd{\alignts@preamble}{\displaystyle}{\textstyle}{}{}
\patchcmd{\alignts@preamble}{\displaystyle}{\textstyle}{}{}

\def\alignts{\let\align@preamble\alignts@preamble\start@align\@ne\st@rredfalse\m@ne}

\makeatother

\allowdisplaybreaks

\title[$C^1$ blow up for the Euler-Poisson System]{Formation of singularities in plasma ion dynamics}

\author[J. Bae]{Junsik Bae}
\address[JB]{Mathematics Division, National Center for Theoretical Sciences, No. 1, Sec. 4, Roosevelt Rd., Taipei 10617, Taiwan}
\email{jsbae@ncts.ntu.edu.tw}

\author[J. Choi]{Junho Choi}
\address[JC]{Institute of Basic Science, Sungkyunkwan University, Gyeonggi-do, 16419, Korea}
\email{junho.choi@skku.edu}
 
\author[B. Kwon]{Bongsuk Kwon}
\address[BK]{Department of Mathematical Sciences, Ulsan National Institute of Science and Technology, Ulsan, 44919, Korea}
\email{bkwon@unist.ac.kr}

\date{\today}


\begin{document}
 
\maketitle 

\begin{abstract}
We study the formation of singularity for the  Euler-Poisson system equipped with the Boltzmann relation, which describes the dynamics of ions in an electrostatic plasma. 
In general, it is known that smooth solutions to nonlinear hyperbolic equations fail to exist globally in time. We establish criteria for $C^1$ blow-up of the  Euler-Poisson system, both for the isothermal and pressureless cases. In particular, our blow-up condition for the presureless model does not require that the gradient of velocity is negatively large. In fact, our result particularly implies that the smooth solutions can break down even if the gradient of initial velocity is trivial. For the isothermal case, we prove that smooth solutions leave $C^1$ class in a finite time when the gradients of the Riemann functions are initially large. 
%

\end{abstract}

\section{Introduction}

We consider the one-dimensional Euler-Poisson system  in a non-dimensional form:
\begin{subequations}\label{EP}
\begin{align}[left = \empheqlbrace\,]
& \rho_t +  (\rho u)_x = 0, \label{EP_1} \\ 
& \rho(u_t  + u u_x) + K\rho_x = -  \rho\phi_x, \label{EP_2} \\
& - \phi_{xx} = \rho - e^\phi. \label{EP_3} 
\end{align}
\end{subequations} 
Here $\rho>0$, $u$ and $\phi$ are the unknown functions of $(x,t) \in \mathbb{R}\times \mathbb{R}^+$ representing the ion density, the fluid velocity for ions, and  the electric potential, respectively. $K = T_i/T_e \geq  0$ is a constant of the ratio of the ion temperature $T_i$ to the electron temperature $T_e$. The system \eqref{EP} is referred to as the \textit{isothermal} model when $K>0$, and the \textit{pressureless} model when $K=0$, respectively.

The Euler-Poisson system \eqref{EP} is a fundamental fluid model describing the dynamics of ions in an electrostatic plasma \cite{Ch,Dav,Pecseli}, and   it is often employed to study various phenomena of plasma such as plasma sheaths \cite{hk, suzuki} and plasma solitons \cite{BK2,BK,HS}. Especially, to study plasma waves, the limit problems seeking the  connections with some well-known dispersive models have been investigated, for instance, KdV limit \cite{BK2, Guo,HNS,LS}, KP-II and Zakharov-Kuznetsov limits \cite{LLS, Pu}, and NLS limit \cite{PuNLS}.  

The Euler-Poisson system \eqref{EP} is  the one-fluid model of ions, where the electron density $\rho_e$ is assumed to satisfy the \textit{Boltzmann relation}
\begin{equation*}\label{Boltzmann}
\rho_e=e^{\phi}.
\end{equation*}
Based on the physical fact that 
the electron mass $m_e$ is much lighter than the ion mass $m_i$, i.e, $m_e/m_i \ll 1$, 
the relation  can be formally derived from the two-fluid model of ions and electrons by suppressing the constant of electron mass $(m_e=0)$, referred to as the \textit{massless electron} assumption.
We refer to \cite{Ch} for more details of physicality and derivation, and also to \cite{GGPS} for a mathematical justification of the massless electron limit.


Due to the nature of electrically reactive fluids, plasmas exhibit unique phenomena different from the usual gas.
Correspondingly, the Euler-Poisson system \eqref{EP}, where electrical effect is described by the Poisson equation with the Boltzmann relation, exhibits interesting and rich dynamics, significantly different from that of the compressible Euler equations. One of the most interesting feathures is that \eqref{EP} admits special types of  solutions such as traveling solitary waves, \cite{Cor,LS,Sag}, whose   linear stability has been studied in \cite{HS} and \cite{BK} for the pressureless case and for the isothermal case, respectively. 
As far as existence of smooth solutions is concerned, while in general smooth solutions to nonlinear hyperbolic equations fail to exist globally in time,  it is interesting  that this special solution can persist globally.
A question of global existence or finite time blow-up of smooth solutions naturally arises in the study of large-time dynamics of the Euler-Poisson system \eqref{EP}, including nonlinear stability of the solitary waves.

In the present paper, we investigate formation of singularities for the 1D Euler-Poisson system \eqref{EP}. For the isothermal case, i.e., $K>0$, we show that smooth solutions to \eqref{EP}  develop $C^1$ blow-up in a finite time when the gradients of Riemann functions are initially large. When the blow-up occurs, we find that the density and velocity stay bounded, while their derivatives blow up. For the pressureless case, i.e., $K=0$, we propose a condition  for formation of singularities, requiring no largeness of the gradient of velocity. It is known that if the initial velocity has negatively large gradient at some point, the smooth solution to the pressureless system \eqref{EP} leaves  $C^1$ class in a finite time, \cite{Liu}.
In contrast, our condition does not require the large gradient of the initial velocity. In particular, our result demonstrates that the density and the derivative of velocity blow up even if the initial velocity has trivial gradient. In fact, it is the electric potential that induces development of singularities. 
For instance, when the initial local density is sufficiently lower than the background density, i.e., ion density is sufficiently rarefied, the electrostatic potential is determined by the distribution of ions in a way that the fluid momentum with negative gradient is generated at later times, resulting in the finite-time singularity formation. We refer to  \cite{PHGOA} for a relevant numerical study for the pressureless Euler-Poisson system. 
We present several numerical experiments supporting our results in Section~\ref{numerical},
where we also provide numerical examples showing that the pressureless model and the isothermal model exhibit the radically different behaviors in the solutions, see Table \ref{Table2}.

In the literature of plasma physics, the isothermal Euler-Poisson system is the most common and important. Yet the pressureless Euler-Poisson system, i.e.,  \eqref{EP} with $K=0$, is often considered as a simplified model for ions in a certain physical situation where the ion temperature $T_i$ is much smaller than the electron temperature $T_e$.
In other words, the pressureless Euler-Poisson system  is an ideal model for \textit{cold ions} (a plasma with  $T_i/T_e \ll 1$).
From a mathematical point of view, the pressureless model is weakly coupled (the hyperbolic part is decoupled) so that one can exploit its simpler structure in the analysis. However,  the presence of the pressure makes the hyperbolic part of \eqref{EP} strongly coupled, which makes it harder to mathematically analyze.  Not suprisingly, properties of solutions to the isothermal model are significantly different from those to the pressureless model in certain regimes. We shall discuss these issues in detail, in  particular, in terms of examples of the blow-up solutions and solitary waves in Section~\ref{numerical}.

To the best of our knowledge, there is no result on the global well-posedness of smooth solutions to the Euler--Poisson system with the Boltzmann relation for the 1D and 2D cases. In fact, global existence of weak entropy solutions for the 1D isothermal case is proved in \cite{CP}. For the 3D isothermal case, smooth irrotational flows can exists globally in time, \cite{GP}. We remark that 
our numerical experiments demonstrate that some smooth solutions converge to a background constant state $(\rho, u, \phi)=(1,0,0)$ as time goes by, 
see Figure~\ref{Fig4}. 
If the smooth solution exists globally in time, one can further ask whether the solution scatters to the constant state. This can be conjectured by the dispersion relation of the associated linear system, 
\begin{equation*}\label{dispersion}
\omega(\xi)
=  \pm i \xi \sqrt{K+ \frac{1}{1+\xi^2}}.
\end{equation*}
The questions of global existence of smooth solutions and their long time behavior are intriguing and challenging  since the system is \textit{weakly dispersive}.

\subsection{Main results}

We consider the Euler-Poisson system \eqref{EP} around a constant state, i.e.,
\begin{equation}\label{Farfield1}
(\rho,u,\phi)(x,t) \to (1,0,0) \quad \text{as } |x| \to \infty.
\end{equation}
We remark that any constant state $(\rho_*, u_*, \phi_*)$ can be normalized into $(\rho_*, u_*, \phi_*)= (1,0,0)$ due to the Galilean transformation for the velocity, normalization of density, and eletrostatic potential reference determined by the density  $\phi_* = \ln \rho_*$ with $\rho_*\ne 0$.

The system \eqref{EP}--\eqref{Farfield1} admits a unique smooth solution locally in time for sufficiently smooth initial data, see \cite{LLS}.\footnote{For instance, $(\rho_0-1,u_0)\in H^2(\mathbb{R})\times H^3(\mathbb{R})$ when $K=0$, and $(\rho_0-1,u_0)\in H^2(\mathbb{R})\times H^2(\mathbb{R})$ when $K>0$.} Furthermore, as long as the smooth solution exists, the energy 
\begin{equation}\label{H-def}
H(t):= \int_\mathbb{R} \frac{1}{2}\rho u^2 + P(\rho)  + \frac12 |\partial_x\phi|^2 +(\phi-1)e^\phi + 1 \,dx,
\end{equation}
where 
\begin{equation*}
P(\rho):=K(\rho\ln\rho - \rho + 1), \quad (K \geq 0),
\end{equation*}
is conserved, that is,
\begin{equation}\label{EnergyConser}
H(t)=H(0).
\end{equation}
Here we note that when $K>0$, the {\it{relative pressure}}  $P(\rho)$ verifies that
\begin{equation}\label{pos_pressure}
P(\rho)  >0 \text{ for } \rho \in (0,1)\cup (1,\infty),
\end{equation} 
and that $P(1)=P'(1)=0$ and $P''(\rho)= K\rho^{-1}>0$.

\subsubsection{Isothermal Case}
To state our first theorem, we introduce the Riemann functions \cite{Rie} associated with the isothermal Euler equations:
\begin{subequations}\label{RI} 
\begin{align}
& r = r(\rho,u) := u  + \int_1^\rho \frac{\sqrt{p'(\xi)}}{\xi}\,d\xi = u + \sqrt{K}\ln \rho,  \\
& s = s(\rho,u) := u -  \int_1^\rho \frac{\sqrt{p'(\xi)}}{\xi}\,d\xi = u - \sqrt{K}\ln \rho,
\end{align}
\end{subequations}
where $p(\rho)$ be the pressure term in \eqref{EP}, i.e., $p(\rho) := K \rho$. 
We note that the solution to \eqref{EP}--\eqref{Farfield1} satisfies that 
\begin{equation}\label{RI-infty}
( r, s )(x,t) \to (0, 0) \quad \text{as } |x| \to \infty.
\end{equation}
In what follows, let $(r_0, s_0)(x) := (r,s)(x,0)$. 
\begin{theorem}[Isothermal case, $K>0$]\label{MainThm_Warm}
For any given positive numbers $T_0$ and $\veps$, there exist $\delta_0(T_0,\veps) = \delta_0 \in (0,\veps)$ and $M(T_0,\delta_0)=M>0$ such that for all $\delta\in(0,\delta_0)$, the following statement holds: if 
\begin{subequations}\label{Thm_Con} 
\begin{align}
& \sup_{x\in\mathbb{R}} |\rho_0(x)-1| \leq \delta ,  \label{Thm_Con_1}\\
& \sup_{x\in\mathbb{R}}|u_0(x)| \leq  \delta, \label{Thm_Con_2}\\
& H(0) \leq \delta,  \label{Thm_Con_3}\\
& -\rho_0^{-1/2}(x)\partial_x r_{0}(x) \geq M \text{ or }  -\rho_0^{-1/2}(x)\partial_x s_{0}(x) \geq M \quad \text{for some } x\in\mathbb{R}, \label{Thm_Con_4}
\end{align}
\end{subequations} 
then the maximal existence time $T_\ast$ of the classical solution to the isothermal Euler-Poisson system \eqref{EP}  satisfying \eqref{Farfield1} does not exceed $T_0$.
Moreover, it holds that 
\begin{equation}\label{gradient-blowup-u}
    \| ( \partial_x  \rho , \partial_x u) (\cdot,t) \|_{L^\infty(\mathbb{R})}    \nearrow  \infty \quad \text{ as } t\nearrow T_\ast 
\end{equation}
while 
\begin{equation}\label{rs-l-infty-bd}
 \sup_{t\in[0, T_\ast)} \| ( \rho , u, \phi, \partial_x \phi, \partial_x^2 \phi ) (\cdot,t) \|_{L^\infty(\mathbb{R})} <\infty. 
\end{equation}
\end{theorem}
Theorem \ref{MainThm_Warm} indicates that smooth solutions to the isothermal model \eqref{EP} develop $C^1$ blow-up in a finite time when the initial state $(\rho_0,u_0)$ {\it{near the electrically neutral regime}} has relatively large gradient.
 In fact,  the condition \eqref{Thm_Con_3} with small $\delta>0$, i.e.,  $H(0)$ being small, implies $\phi\approx 0$ initially (see Lemma \ref{phi-bd}). 
We note that $H(0)$ is controlled by $\| (\rho_0-1, u_0)\|_{L^2}$ due to the elliptic estimates for the Poisson equation \eqref{EP_3} (see Section \ref{Appen1}): 
\begin{equation}\label{EnergyBd-0}
0 \leq H(0) \leq \frac{\sup_{x \in \mathbb{R}}\rho_0}{2} \int_{\mathbb{R}} |u_0|^2\,dx + (\frac{1}{\kappa_0}+C\delta) \int_{\mathbb{R}} |\rho_0-1|^2\,dx,
\end{equation}
where
 $\kappa_0:= (1-\inf\rho_0)/(-\log \inf\rho_0)$.
When the singularity occurs, we find that the density and velocity stay bounded, while their derivatives blow up. This is one of the interesting features when the pressure is present, while the pressureless case exhibits the blow-up of $L^\infty$ norm of the density. We present some numerical experiments supporting our result in Section~\ref{numerical}, see Figure \ref{Fig3} and \ref{Fig2}. 



Along the characteristics associated with the distinct eigenvalues of the hyperbolic part of \eqref{EP},
\begin{equation}\label{Eigen}
\lambda^+ = \lambda^+(\rho,u) :=  u + \sqrt{K}, \quad \lambda^- = \lambda^-(\rho,u) :=  u - \sqrt{K},
\end{equation}
the corresponding Riemann functions \eqref{RI} satisfy
\begin{equation}\label{RI_1}
r' =  -\phi_x, \quad s^\bp  =   -\phi_x,
\end{equation}
where 
\[
' := \partial_t + \lambda^+ \partial_x, \quad  ^\bp :=\partial_t + \lambda^- \partial_x,
\] 
respectively.
Following the elegant calculation of Lax \cite{Lax}, we obtain that
\begin{subequations}\label{RI_1a}
\begin{align}
& (-\rho^{-1/2}r_x)' - \rho^{1/2}\frac{(\rho^{-1/2}r_x)^2}{2} =    \rho^{-1/2}\phi_{xx} = \rho^{-1/2}(e^\phi-\rho), \\
& (-\rho^{-1/2}s_x)^\bp - \rho^{1/2} \frac{(-\rho^{-1/2}r_x)^2}{2} = \rho^{-1/2}\phi_{xx} = \rho^{-1/2}(e^\phi-\rho). 
\end{align}
\end{subequations} 
For the Euler equations,
it is a well known result of \cite{Lax} that if $r_x$ or $s_x$ is initially negative at some point, $\rho_x$ and $u_x$ will blow up in a finite time.
However,  for the case of \eqref{EP}, where the non-local effect due to the Poisson equation comes into play, the Riemann functions are not conserved along the characteristics so that the forementioned blow-up analysis for the Euler equations  is no longer applicable. 
To resolve this issue, we borrow the idea developed in \cite{Daf1} to keep track of the time-evolution of the $C^1$ norms of the Riemann functions along the characteristics. A similar approach is also adopted in \cite{Daf2,DH,WC}, and we refer to \cite{WC} for the Euler-Poisson system with heat diffusion and damping relaxation, which governs  electron dynamics with a fixed background ion. 

In our analysis,  we  obtain  the uniform bounds for $\phi$ and $\phi_x$  by making use of the energy conservation. More precisely, we first show that the amplitude of $\phi$ is bounded \textit{uniformly in $x$ and $t$} as long as the smooth solution exists (Lemma \ref{phi-bd}) and that this uniform bound can be controlled only by the size of initial energy $H(0)$. With the aid of the convexity of $P(\rho)$, this fact further implies that the uniform bound for $\phi_x$ is also controlled by the initial energy $H(0)$ (Lemma \ref{Lemma_P1}). We remark that in contrast to the proof of Lemma \ref{phi-bd}, the proof of Lemma \ref{Lemma_P1} relies on the fact that $K>0$.

\subsubsection{Pressureless Case}
To state our second theorem, let us define a function $V_-: (-\infty,0] \to [0,\infty)$ by 
\[
V_-(z):= \int_z^0 \sqrt{2\left((\tau-1)e^\tau + 1\right)}\,d\tau \; \text{ for }  z \in (-\infty,0].
\] 
By inspection, we see that  $V_-$ is well-defined since $(\tau-1)e^\tau + 1$ is nonnegative, it is strictly decreasing in $(-\infty,0]$, and  it has the inverse function $V_-^{-1}:[0,+\infty) \to (-\infty,0]$.
\begin{theorem}[Presssureless case, $K=0$]\label{MainTheorem}
For the initial data satisfying
\begin{equation}\label{ThmCon2}
exp\left(V_-^{-1}(H(0))\right) > 2\rho_0(\alpha) \text{ for some } \alpha \in \mathbb{R},
\end{equation}
 the maximal existence time $T_*$ of the classical solution to the pressureless Euler-Poisson system \eqref{EP} satisfying \eqref{Farfield1} is finite. In particular, it holds that
\[
\lim_{t \nearrow T_\ast}\sup_{x \in \mathbb{R}}\rho(x,t) = +\infty \quad \text{ and } \quad \inf_{x \in \mathbb{R}}u_x(x,t) \approx \frac{1}{t-T_\ast}
\]
for all $t<T_\ast$ sufficiently close to $T_\ast$.
\end{theorem}
Theorem \ref{MainTheorem} demonstrates that singularities in solutions to the pressureless model \eqref{EP} can occur in a finite time if the initial density at some point is small compared to the initial energy. 
In fact, the negativity of the initial velocity gradient is not required.

We remark that there is a fairly wide class of the initial data satisfying the condition \eqref{ThmCon2}. From the elliptic estimates for the Poisson equation \eqref{EP_3}, we have (see Section \ref{Appen1})
\begin{equation}\label{EnergyBd}
0 \leq H(0) \leq \frac{\sup_{x \in \mathbb{R}}\rho_0}{2} \int_{\mathbb{R}} |u_0|^2\,dx + \frac{1}{K_0} \int_{\mathbb{R}} |\rho_0-1|^2\,dx =: C(\rho_0,u_0),
\end{equation}
where
 $K_0:= (1-\inf\rho_0)/(-\log \inf\rho_0)$.
On the other hand, since $\lim_{\zeta \searrow 0} V_-^{-1}(\zeta) = 0$, for any given constant $0<c<1/2$, there is $\delta_c>0$ such that $\zeta<\delta_c$ implies $\exp(V_-^{-1}(\zeta))>2c$. Thus, \eqref{ThmCon2} holds for all initial data satisfying $\inf \rho_0=c \in (0,1/2) $ and $C(\rho_0,u_0)<\delta_c \ll 1$. In particular, one can take $u_0 \equiv 0$. 


For the pressureless case, along the characteristic curve $x(\alpha,t)$ associated with the fluid velocity $u$, issuing from an initial point $\alpha \in \mathbb{R}$ (see \eqref{CharODE}), one can easily obtain from \eqref{EP} that
\begin{equation}\label{Diff_Eq_1}
D\rho/Dt = -u_x \rho, \quad Du_x/Dt = -u_x^2 + \rho - e^\phi, \quad (D/Dt:= \partial_t + u \partial_x).
\end{equation}
 The behavior of $\rho$ and $u_x$ depends not only on the initial data, but the potential $\phi$ along the characteristic curve due to the nonlocal nature of the system \eqref{EP}. 
 
In \cite{Liu}, one sufficient condition for blow-up was obtained by discarding $e^\phi$ in \eqref{Diff_Eq_1} and solving the resulting (closed) system of differential inequalities for $\rho$ and $u_x$. The solution blow up if the initial data satisfies $\partial_x u_0 \leq -\sqrt{2\rho_0}$ at some point, i.e., the gradient of velocity is large negatively compared to the density.  
 On the other hand, our analysis takes account of the non-local effect, and as such, the blow up criterion \eqref{ThmCon2} involves the non-local quantity.


As in the isothermal case, one can invoke the energy conservation to show that $\phi$ is uniformly bounded in $x$ and $t$ (Lemma \ref{phi-bd}).
Next, we define  
\[
w(\alpha,t):= \frac{\partial x}{\partial \alpha}(\alpha,t)
\]
and derive a second-order ODE \eqref{2ndOrdODE} for $w$. Using Lemma \ref{phi-bd}, we find that $w$ vanishes at a finite time $T_\ast$ if and only if the solution blows up in the $C^1$ topology, i.e., $u_x\searrow -\infty$ as $t \nearrow T_\ast$ at a non-integrable order in time $t$ (Lemma~\ref{Lem_Blowup}).  Our goal is then to find some sufficient conditions guaranteeing $w$ vanishes in a finite time. By applying Lemma \ref{phi-bd}, we employ a comparison argument for the differential inequality to study the behavior of $w$.

The derivation of \eqref{2ndOrdODE} is related to the well-known fact that the Riccati equation can be reduced to a second-order linear ODE (\cite{Ince}, pp.23–25). The Lagrangian formulation is also adopted for some simplified  Euler-Poisson systems, for instance, the ones with zero background \cite{HJL} and constant background \cite{Dav2}. Due to the absence of the Boltzmann relation, the ODE systems for these models corresponding to \eqref{Diff_Eq_1} do not involve the nonlocal term, and one obtains exact solutions of the associated ODEs. (See also Chapter 3 in \cite{Dav} or  p.301 in \cite{Pecseli}.)  The works of \cite{CCTT, ELT,LT1, LT} study the so-called critical threshold for some types of the pressureless Euler-Poisson systems. An interesting open question is whether such critical threshold exists for the pressureless Euler-Poisson system with the Boltzmann relation.

%
%

The paper is organized as follows. In Section \ref{Sect2.1}, we prove the uniform  bounds of $\phi$ and $\phi_x$ in $x$ and $t$. Theorem \ref{MainThm_Warm} and Theorem \ref{MainTheorem} are proved in Section \ref{Sect2.2} and Section \ref{Sect2.3}, respectively. In Section \ref{numerical}, we present several numerical experiments supporting our results as well as numerical examples in which the solutions to the pressureless model and the isothermal model behave differently.

\section{Proof of  Main Theorems}\label{Sect2}
This section is devoted to the proof of our main theorems. We first present some preliminary lemmas that will be crucially used later. We establish the uniform  bounds of $\phi$ and $\phi_x$ in $x$ and $t$.
\subsection{Uniform bounds of $\phi$ and $\phi_x$.}\label{Sect2.1}
 Let us define the functions  
\begin{equation*}
V(z):=
\left\{
\begin{array}{l l}
V_+(z):= \displaystyle{ \int_0^z \sqrt{2U(\tau)}\,d\tau } \; \text{ for } z \geq 0, \\ 
V_-(z):= \displaystyle{ \int_z^0 \sqrt{2U(\tau)}\,d\tau } \; \text{ for }  z \leq  0,
\end{array} 
\right.
\end{equation*}
where $U(\tau):=(\tau-1)e^\tau + 1$ is nonnegative for all $\tau\in \mathbb{R}$ and satisfies 
\[
U(\tau) \to +\infty \text{ as } \tau \to +\infty, \quad U(\tau)\to  1 \text{ as } \tau \to -\infty
\]
(see Figure \ref{U&f}). Hence, $V_+$ and $V_-$ have the inverse functions $V_+^{-1}:[0,+\infty) \to [0,+\infty)$ and $V_-^{-1}:[0,+\infty) \to (-\infty,0]$, respectively. Furthermore, $V$ is of $C^2(\mathbb{R})$. 
\begin{lemma} \label{phi-bd}
As long as the smooth solution to \eqref{EP}--\eqref{Farfield1} exists for $t\in[0,T]$, it holds that
\[
V_-^{-1}\left( H(0) \right) \leq \phi(x,t) \leq V_+^{-1}\left( H(0) \right) \quad \text{ for all } (x,t) \in \mathbb{R} \times [0,T].
\]
\end{lemma}
\begin{proof}
Since $V \in C^1(\mathbb{R})$ and $V \geq 0$,  we have that for all $t \geq 0$ and $x\in\mathbb{R}$,
\begin{equation*}
\begin{split}
0\leq V \left(\phi(x,t) \right) 
& = \int_{-\infty}^x \frac{dV}{dz}(\phi(y,t))\phi_y\,dy  \\
&  \leq \int_{-\infty}^x \left| \frac{dV}{dz}(\phi(y,t))\right| |\phi_y|\,dy \\
&  \leq  \int_{-\infty}^\infty U(\phi)\,dy + \frac{1}{2} \int_{-\infty}^\infty |\phi_y|^2\,dy \\
& \leq \int_\mathbb{R} \frac{1}{2}\rho u^2 + K(\rho\ln\rho-\rho+1) + \frac12 |\partial_x\phi|^2 +(\phi-1)e^\phi + 1 \,dx \\
& = H(t) = H(0),
\end{split}
\end{equation*}
where the last equality holds due to the energy conservation \eqref{EnergyConser} and the second to the last inequality holds due to \eqref{pos_pressure}. This completes the proof.
\end{proof}

\begin{lemma}\label{Lemma_P1}
Let $K>0$. Assume that  $|\rho(x,t)-1|<1$ for all $(x,t) \in \mathbb{R} \times [0,T]$. Then there is a constant $C_0(K)>0$ such that   
\begin{equation}\label{Lemma_P1_eq}
|\phi_x (x,t)|^2 \leq C_0(K) \cdot O\big(|H(0)|\big) \quad \text{as} \quad |H(0)| \to 0
\end{equation}
for all $(x,t) \in \mathbb{R} \times [0,T]$.
\end{lemma}
\begin{proof}
Multiplying the Poisson equation by $-\phi_x$, and then integrating in $x$, 
\begin{equation}
\begin{split}\label{phi-x-2}
\frac{\phi_x^2}{2} 
& = \int_{-\infty}^x(\rho-1)(-\phi_x)\,dx + \int_{-\infty}^x(e^\phi-1)\phi_x\,dx \\
& \leq \frac{1}{2}  \int_{\mathbb{R}} | \rho-1 |^2 dx + \frac{1}{2} \int_{\mathbb{R}} | \phi_x|^2 dx + e^\phi - \phi -1 \\ 
& \le \frac{1}{2} \int_{\mathbb{R}} | \rho-1 |^2 dx  + O\big(|H(0)|\big) \quad \text{as} \quad |H(0)| \to 0.
\end{split}
\end{equation}
Here we have used the fact that 
\[ \frac12 \int_{\mathbb{R}} |\phi_x |^2 dx \le H(t) = H(0)\] 
and, by the Taylor expansion with Lemma~\ref{phi-bd}, that 
\[ e^\phi - 1 - \phi \le O\big(|H(0)|\big) \quad  \text{as} \quad |H(0)| \to 0.\]
As long as $|\rho(x,t)-1|<1$ for all $(x,t) \in \mathbb{R} \times [0,T]$, it holds that 
\begin{equation}\label{phi-x-3}
\int_{-\infty}^\infty \frac{1}{4}|\rho-1|^2\,dx \leq \int_{-\infty}^\infty \rho \ln \rho - \rho +1 \,dx   \leq  \frac{H(t)}{K} = \frac{H(0)}{K}
\end{equation}
for all $t\in[0,T]$.
The first inequality in \eqref{phi-x-3} holds thanks to the Taylor expansion, and the second inequality in \eqref{phi-x-3} holds due to \eqref{H-def}. Combining \eqref{phi-x-2} and \eqref{phi-x-3}, we obtain \eqref{Lemma_P1_eq}. We are done. 
\end{proof}
\begin{figure}[h]
\resizebox{110mm}{!}{\includegraphics{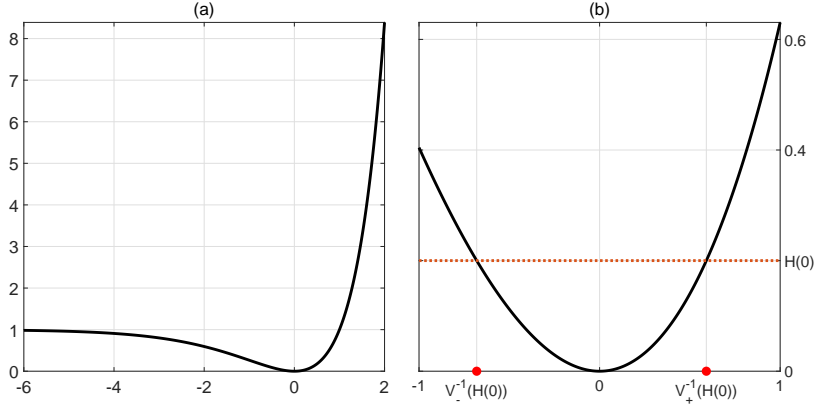}}
\caption{(a): The graph of $U(\tau)=(\tau-1)e^\tau +1$.  (b): The   graph of $V(z)$. By Lemma \ref{phi-bd}, $\phi$ is confined in the interval $[V_{-}^{-1}(H(0)), V_{+}^{-1}(H(0))]$. } \label{U&f} \end{figure}
Now we are ready to prove the main theorems. 
\subsection{Proof of Theorem \ref{MainThm_Warm}}\label{Sect2.2}
Let
\begin{equation}
W := r_x, \quad Z := s_x.
\end{equation}
 By taking $\partial_x$ of \eqref{RI}, we have
\begin{equation}\label{RI_2}
u_x = \frac{W+Z}{2}, \quad \rho_x = \frac{\rho(W-Z)}{2\sqrt{K}}.
\end{equation}
Taking $\partial_x$ of \eqref{RI_1}, and then using \eqref{Eigen} and \eqref{RI_2}, we get
\begin{align}\label{RI_3}
\begin{split}
\begin{split}
-\phi_{xx} 
& = W' + \lambda^+_\rho \rho_x W + \lambda^+_u u_x W \\
& = W' + \frac{W^2}{2} + \frac{ZW}{2},
\end{split} \\
\begin{split}
-\phi_{xx} 
& = Z^\bp + \lambda^-_\rho \rho_x Z + \lambda^-_u u_x Z \\
& = Z^\bp + \frac{Z^2}{2} + \frac{ ZW}{2}.
\end{split}
\end{split}
\end{align}
On the other hand, from \eqref{EP_1}  and \eqref{RI_2}, we have
\begin{equation}\label{RI_4}
\rho' = -\rho Z, \quad \rho^\bp = -\rho W.
\end{equation}
Multiplying \eqref{RI_3} by the integrating factor $\rho^{-1/2}$, and then using \eqref{RI_4} and the Poisson equation \eqref{EP_3}, we obtain 
\begin{subequations} \label{RI_6} 
\begin{align}
& f' = \rho^{1/2}\frac{f^2}{2}  + \rho^{-1/2}(e^\phi - \rho), \\
& g^\bp = \rho^{1/2} \frac{g^2}{2} + \rho^{-1/2}(e^\phi - \rho),
\end{align}
\end{subequations}
where 
\begin{equation*}
f := -\rho^{-1/2}W, \quad g:= -\rho^{-1/2}Z.
\end{equation*}


We define the Lipschitz functions $R(t)$ and $S(t)$ on $[0,T]$ by 
\begin{equation}\label{Def_RS}
R(t) := \max_{x\in \mathbb{R}}|r(x,t)|, \quad  S(t)  := \max_{x\in \mathbb{R}}|s(x,t)|. 
\end{equation}
Here $R(t)$ and $S(t)$ exist as long as the smooth solution to \eqref{EP} satisfies \eqref{RI-infty}. 

Let $\veps \in (0,\tfrac{1}{4})$ be a given number. There is $T_1>0$ such that the solution to \eqref{EP} with the initial data satisfying \eqref{Thm_Con_1}--\eqref{Thm_Con_3} (with $\delta<\veps$) satisfies that for all $(x,t)\in \mathbb{R}\times [0,T_1]$,
\begin{equation}\label{AP1}
|\rho(x,t) - 1| \leq  2\veps.
\end{equation}
We fix $t\in [0,T_1)$ and choose points $\hat{x},\check{x} \in \mathbb{R}$ such that 
\begin{equation*}
R(t) = |r(\hat{x},t)|, \quad S(t) = |s(\check{x},t)|.
\end{equation*}
For any $h \in (0,t)$, we have from \eqref{Def_RS} that
\begin{align*}
\begin{split}
R(t - h) \geq |r(\hat{x} - h \lambda^+(\rho(\hat{x},t),u(\hat{x},t)) ,t - h)|, \\
S(t - h) \geq |s(\check{x} - h \lambda^-(\rho(\check{x},t),u(\check{x},t)) ,t - h)|.
\end{split}
\end{align*}
Then, it is straightforward to check that 
\begin{equation}\label{AP_R1} 
\lim_{h \to 0^+} \frac{R(t-h) - R(t)}{-h} \leq |r'|(\hat{x},t),
\end{equation} 
provided that the limit on the LHS of \eqref{AP_R1} exists. Indeed, if $r(\hat{x},t) \neq 0$, then 
\[
\begin{split}
\lim_{h \to 0^+} \frac{R(t-h) - R(t)}{-h} 
& \leq |r|'(\hat{x},t) \\
& = \frac{r(\hat{x},t)}{|r|(\hat{x},t)}r'(\hat{x},t) \\
& \leq |r'|(\hat{x},t),
\end{split}
\]
and if $r(\hat{x},t)=0$, then
\[
\begin{split}
\lim_{h \to 0^+}\frac{R(t-h) - R(t)}{-h}
& = \lim_{h \to 0^+}\frac{R(t-h) }{-h} \\
& \leq 0 \\
& \leq |r'|(\hat{x},t).
\end{split}
\]
In a similar fashion, we have that 
\begin{equation}\label{AP_R1_1}
\lim_{h \to 0^+} \frac{S(t-h) - S(t)}{-h} \leq |s^\bp|(\check{x},t),
\end{equation}
provided that the limit on the LHS of \eqref{AP_R1_1} exists.  Using \eqref{RI_1}, Lemma~\ref{Lemma_P1} and \eqref{Thm_Con_3}, we obtain from \eqref{AP_R1} and \eqref{AP_R1_1} that  there is a constant $C_1>0$ such that
\begin{equation}\label{AP_R2}
\begin{split}
\frac{d}{dt}\left( R(t)+ S(t) \right) 
& \leq |r'|(\hat{x},t) + |s^\bp|(\check{x},t)  \\
& \leq 2\max_{x}|\phi_x(\cdot,t)| \\
& \leq C_1 \delta^{1/2}
\end{split}
\end{equation}
for almost all $t\in [0,T_1]$. Integrating \eqref{AP_R2} in $t$, we get 
\begin{equation}\label{AP_R3}
R(t) + S(t)  \leq R(0) + S(0) + t  C_1 \delta^{1/2}
\end{equation}
for all $t\in[0,T_1]$. Then, from \eqref{RI}, \eqref{Def_RS} and \eqref{AP_R3}, we notice that
\begin{equation}\label{AP11}
\begin{split}
|\rho-1|
& =|\exp\left(\frac{r-s}{2\sqrt{K}} \right) - 1| \\
& \leq \exp \left| \frac{r-s}{2\sqrt{K}} \right| - 1  \\
& \leq \exp \left( \frac{R(0) + S(0) + t  C_1\delta^{1/2}}{2\sqrt{K}}\right) - 1 \\
& \leq \exp \left( \frac{2\max|u_0| + 2\sqrt{K}\max|\ln\rho_0| +  t C_1\delta^{1/2} }{2\sqrt{K}}\right) - 1.
\end{split}
\end{equation} 

For any given $T_0>0$, we choose $\delta_0=\delta_0(T_0,\veps) \in (0,\veps)$ sufficiently small such that for all $\delta\in(0,\delta_0)$, it holds that
\begin{equation}\label{AP11_1}
\exp \left( \frac{2\max|u_0| + 2\sqrt{K}\max|\ln\rho_0| +  T_0 C_1\delta^{1/2} }{2\sqrt{K}}\right) - 1 < \veps,
\end{equation}
provided that $\max|u_0|<\delta$ and $\max|\rho_0-1| <\delta$. 
Let $\alpha$ and $\beta$ be the numbers such that 
\begin{equation}\label{AP2}
2\beta \geq \rho^{1/2} \geq 2\alpha >0  \quad \textrm{ for all } \rho \in [1-\veps,1+\veps]
\end{equation}
holds. For instance, let $\alpha=\frac{\sqrt{1-\veps}}{2 }$ and $\beta = \frac{\sqrt{1+\veps} }{2 }$.

Now, with $\delta \in (0,\delta_0)$, we solve \eqref{EP} with the initial data satisfying \eqref{Thm_Con_1}--\eqref{Thm_Con_3}. Let $T_\ast$ be  the maximal existence time for the classical solution. We suppose to the contrary that $T_0<T_\ast$.

Using the continuity argument, we claim that 
\begin{equation}\label{AP1_1}
|\rho(x,t) - 1| \leq  2\veps \quad \text{ for all } (x,t)\in\mathbb{R}\times[0,T_0].
\end{equation}
We define the continuous function $Y(t):=\sup_{0 \leq s \leq t}\sup_{x \in \mathbb{R}}|\rho(x,s)-1|$. Then, $Y(0)\leq \delta<\veps$. 
Suppose to the contrary that $Y(t) > 2\veps$ for some $t\in[0,T_0]$. Then  by continuity, there is $t_0\in[0,T_0]$ such that $Y(t_0)=2\veps$ and 
\[
|\rho(x,t)-1| \leq 2\veps \quad \text{ for all } (x,t)\in\mathbb{R}\times[0,t_0].
\]
Then, from the previous calculation, \eqref{AP11} and \eqref{AP11_1}, we have that 
\begin{equation*}
|\rho(x,t)-1| \leq \veps \quad \text{ for all } (x,t)\in\mathbb{R}\times[0,t_0].
\end{equation*}
Hence we obtain that $2\veps = Y(t_0) \leq \veps $, which is a contradiction. This proves \eqref{AP1_1}.

Let  $C_2 = C_2(\delta_0) := \max\{ e^{ V^{-1}_+(\delta_0)}-1 , 1-  e^{ V^{-1}_-(\delta_0)}  \}>0$. Note that $C_2(\delta_0) \to 0$ as $\delta_0 \to 0$. Let $\gamma=\gamma(\veps) :=  (2\beta)^{-1} ( C_2  + 2\veps)$. Then, we get that
\begin{equation}\label{gam-bd}
\begin{split}
|\rho^{-1/2}(e^\phi - \rho)|
& \leq \rho^{-1/2}\left(|e^\phi -1| + |1 - \rho| \right) \\
& \leq (2\beta)^{-1} ( C_2 + 2\veps)\\
& = \gamma
\end{split}
\end{equation}
for all $(x,t)\in\mathbb{R}\times[0,T_0]$.

Now we choose $M(T_0,\veps)>0$ sufficiently large such that 
\begin{equation}\label{AP12}
 M \geq    \gamma T_0 + \frac{1}{\alpha T_0}.
\end{equation}
We define the Lipschitz functions
\begin{equation}\label{Def_F+G+-0}
F^+(t):=\max_{y\in\mathbb{R}} f(y,t), \ \ G^+(t):=\max_{y\in\mathbb{R}} g(y,t).
\end{equation}
Notice that $F^+(t)$ and $G^+(t)$ exist as long as the smooth solution to \eqref{EP} satisfies the end-state condition \eqref{RI-infty}. In fact,  if $f(x,\cdot)<0$ for all $x \in \mathbb{R}$, then $\partial_x r(x,\cdot)>0$ for all $x \in \mathbb{R}$, which contradicts to \eqref{RI-infty}. Hence, $f(x_0,\cdot)\geq 0$ for some $x_0\in\mathbb{R}$, and it follows from \eqref{RI-infty} that $F^+(t)$ exists.

We fix $t\in[0,T_0]$ and choose points $\hat{y},\check{y} \in \mathbb{R}$ such that 
\begin{equation}\label{Def_F+G+}
F^+(t) = f(\hat{y},t), \quad G^+(t) = g(\check{y},t).
\end{equation}
For any $h \in (0,T_0 - t)$, we have
\begin{align}\label{AP3}
\begin{split}
F^+(t + h) \geq f(\hat{y} + h \lambda_+(\rho(\hat{y},t),u(\hat{y},t)) ,t + h), \\
G^+(t + h) \geq g(\check{y} + h \lambda_-(\rho(\check{y},t),u(\check{y},t)) ,t + h).
\end{split}
\end{align}
Then, it holds that 
\begin{equation}\label{AP_31}
\begin{split}
\lim_{h \to 0^+} \frac{F^+(t+h) - F^+(t)}{h} \geq  f'(\hat{y},t), \\
\lim_{h \to 0^+} \frac{G^+(t+h) - G^+(t)}{h} \geq  g^\bp(\check{y},t),
\end{split}
\end{equation} 
provided that the limit on the LHS of \eqref{AP_31} exists. Using \eqref{RI_6}, \eqref{AP2}, and \eqref{gam-bd}, we get 
\begin{align}\label{AP_32}
\begin{split}
f'(\hat{y},t) \geq \alpha f^2(\hat{y},t)  - \gamma = \alpha (F^+(t))^2  -\gamma , \\
g^\bp(\hat{y},t) \geq \alpha g^2(\check{y},t)  - \gamma = \alpha (G^+(t))^2  -\gamma.
\end{split}
\end{align}
By \eqref{AP_31} and \eqref{AP_32}, it holds that for almost all $t \in[0,T_0]$, 
\begin{align}\label{AP4}
\begin{split}
\frac{d}{dt}F^+(t) \geq \alpha (F^+(t))^2  -\gamma, \\
\frac{d}{dt}G^+(t) \geq \alpha (G^+(t))^2  -\gamma.
\end{split}
\end{align} 
We assume that $f(x,0) \geq M$ for some $x\in\mathbb{R}$, see \eqref{Thm_Con_4}. The other case can be treated similarly. 

Let us define the function
\begin{equation}\label{AP5}
X(t):= F^+(t)  - \gamma(T_0 -t), \quad t \in [0, T_0].
\end{equation}
By \eqref{AP4} and \eqref{AP5}, we see that
\begin{equation}\label{AP6}
\begin{split}
\frac{dX}{dt} \geq \alpha 
& \left( X(t) +  \gamma(T_0-t)\right)^2 
\end{split}
\end{equation}
for almost all $t\in[0,T_0]$. 
From \eqref{AP5}, \eqref{AP12}, \eqref{Thm_Con_4}, we have
\begin{equation}\label{AP12_1}
X(0) \ge M - \gamma T_0 \geq \frac{1}{\alpha T_0}.
\end{equation}
Since $X'(t) \geq 0$ from \eqref{AP6}, it follows from \eqref{AP12_1} that $X(t)>0$ for all $t\in[0,T_0]$, and hence we obtain from \eqref{AP6} that  

\begin{equation}\label{AP7}
\frac{dX}{dt} \geq \alpha  X^2(t)
\end{equation} 
for almost all $t\in[0,T_0]$. Integrating \eqref{AP7}, we have
\begin{equation}\label{AP7_1}
X(t) \geq \frac{X(0)}{1-\alpha X(0) t} \quad \textrm{for all } t\in[0,T_0].
\end{equation}
From \eqref{AP12_1} and \eqref{AP7_1}, we conclude that $X(t)$ blows up for some $t\in[0,T_0]$. This contradicts   the hypothesis that $T_0<T_\ast$. Therefore, $T_\ast\leq T_0$. 

Next we prove the boundedness of the solution. 
Let $x^+(t)$ be the characteristic curve associated with $\lambda^+$ issuing from $x_0$, i.e., 
\[ \dot{x}^+(t) =\lambda^+(r(x^+(t),t), s(x^+(t),t)) , \quad x^+(0) =x_0. \]
Then by \eqref{RI_1}, one has 
\[ \frac{d}{dt} r(x^+(t),t) = -\phi_x (x^+(t), t),\]
and upon integration, we have 
\[ r(x^+(t), t) = r_0(x_0) - \int_0^t \phi_x(x^+(\tau), \tau ) d\tau.\]
Now using Lemma~\ref{Lemma_P1}, we obtain
\[ \| r(\cdot, t) \|_{L^\infty(\mathbb{R})} \le \| r_0\|_{L^\infty(\mathbb{R})} + C T_\ast \delta^{1/2}.\]
The estimate for $s$,
\[ \| s(\cdot, t) \|_{L^\infty(\mathbb{R})} \le \| s_0\|_{L^\infty(\mathbb{R})} + C T_\ast \delta^{1/2},\]
can be obtained in a similar way from \eqref{RI_1}. These together with \eqref{RI} imply 
\begin{equation*} 
 \sup_{t\in[0, T_\ast)} \| ( \rho , u ) (\cdot,t) \|_{L^\infty(\mathbb{R})} <\infty,
\end{equation*}
and thanks  to Lemma~\ref{phi-bd} and Lemma~\ref{Lemma_P1}, 
\begin{equation*}
 \sup_{t\in[0, T_\ast)} \| \partial_x^k \phi (\cdot,t) \|_{L^\infty(\mathbb{R})} <\infty, \quad k=0,1,2.
\end{equation*}
This proves \eqref{rs-l-infty-bd}. 
This completes the proof of Theorem~\ref{MainThm_Warm}. \qed

\begin{remark}[Lower bound of the existence time]
Let us define the Lipschitz functions 
\begin{equation}
F(t):= \max_{x \in \mathbb{R}}|f(x,t)|, \quad G(t):=\max_{x \in \mathbb{R}}|G(x,t)|
\end{equation}
on $[0,T]$. In a similar fashion as the previous calculation, we obtain that
\begin{equation}
\frac{d}{dt}F(t) \leq \beta F^2(t) + \gamma \leq \beta \left( F(t)+\frac{\sqrt{\gamma}}{\sqrt{\beta}}\right)^2.
\end{equation}

Letting $Y(t)=F(t) + \tfrac{\sqrt{{\gamma}}}{{\sqrt{\beta}}}$, then solving the resulting differential inequality for $Y$, we obtain
\[
Y(t) \leq \frac{Y(0)}{1-\beta Y(0)t}.
\]
Hence, the solution exists at least on $[0,T_m)$, where
\[
T_m:= \min\left\lbrace \left[\beta\left( \max_{x\in \mathbb{R}}|f_0| + \frac{\sqrt{\gamma}}{\sqrt{\beta}}\right)\right]^{-1}, \left[\beta\left( \max_{x\in \mathbb{R}}|g_0| + \frac{\sqrt{\gamma}}{\sqrt{\beta}}\right)\right]^{-1} \right\rbrace.
\]
\end{remark}

\subsection{Proof of Theorem~\ref{MainTheorem}}\label{Sect2.3}
This subsection is devoted to the proof of  Theorem~\ref{MainTheorem}.

For $u\in C^1$, the characteristic curves $x(\alpha,t)$ are defined as the solution to the ODE
\begin{equation}\label{CharODE}
\dot{x} = u(x(\alpha,t),t), \quad x(\alpha,0)=\alpha \in \mathbb{R}, \quad t \geq 0,
\end{equation}
where $\dot{•}:=d/dt$ and the initial position $\alpha$ is considered as a parameter. 
Since $x(\alpha,t)$ is differentiable in $\alpha$, we obtain from \eqref{CharODE} that
\begin{equation}\label{var_ODE}
\dot{w} = u_x(x(\alpha,t),t) w, \quad w (\alpha,0)=1, \quad t \geq 0,
\end{equation}
where 
\[
w=w(\alpha,t):= \frac{\partial x}{\partial \alpha}(\alpha,t).
\]

We show that $w$ satisfies a certain second-order ordinary differential equation. By integrating \eqref{EP_2} along $x(\alpha,t)$, we obtain that
\begin{equation}\label{CharODE4}
\dot{x} = u(x(\alpha,t),t) =  u_0(\alpha) - \int_0^t \phi_x(x(\alpha,s),s) \,ds.
\end{equation}
 Differentiating \eqref{CharODE4} in $\alpha$, 
\begin{equation}\label{CharODE5}
\dot{w} =  \partial_\alpha u_{0}(\alpha) - \int_0^t \phi_{xx}(x(\alpha,s),s) w(\alpha,s) \,ds.
\end{equation}
Since the RHS of \eqref{CharODE5} is differentiable in $t$, so is the LHS. Hence, we get
\begin{equation}\label{CharODE6}
\ddot{w} = - \phi_{xx}  w = (\rho- e^\phi) w,
\end{equation}
where we have used \eqref{EP_3}.
On the other hand, using \eqref{EP_1} and \eqref{var_ODE}, we obtain that
\[
\begin{split}
\frac{d}{dt}\left( \rho (x(\alpha,t),t) w(\alpha,t) \right)
& = - \rho u_x w + \rho u_x w = 0,
\end{split}
\]
which yields 
\begin{equation}\label{wrho}
\rho (x(\alpha,t),t) w(\alpha,t)  = \rho_0(\alpha).
\end{equation}

%
%


Finally, combining \eqref{var_ODE}, \eqref{CharODE6},  \eqref{wrho}, we see that $w(\alpha,t)$ satisfies the second-order nonhomogeneous equation
\begin{equation}\label{2ndOrdODE}
\ddot{w} + e^{\phi(x(\alpha,t),t)} w = \rho_0(\alpha), \quad w(\alpha,0)=1 , \quad \dot{w}(\alpha,0) = u_{0x}(\alpha).
\end{equation}




From \eqref{wrho}, it is obvious that  for each $\alpha \in \mathbb{R}$, 
\begin{equation*}
\left.
\begin{array}{l l}
0<w(\alpha,t)<+\infty \quad & \Longleftrightarrow \quad 0<\rho(x(\alpha,t),t)<+\infty, \\ 
\lim_{t \nearrow T_*} w(\alpha,t)=0 \quad  & \Longleftrightarrow \quad \lim_{t \nearrow T_*} \rho(x(\alpha,t),t) = +\infty.
\end{array} 
\right.
\end{equation*}
Using Lemma \ref{phi-bd}, we show that $\sup_{x\in\mathbb{R} } |\rho(x,t)|$ and $\sup_{x\in\mathbb{R} } |u_x(x,t)|$ blow up at the same time, if one of them blows up at a finite time $T_\ast$.

\begin{lemma}\label{Lem_Blowup}
Suppose that the classical solution to \eqref{EP} with $K=0$ exists for all $0 \leq t < T_\ast < +\infty$. Then the following statements hold.
\begin{enumerate}
\item  For each $\alpha \in \mathbb{R}$, the following holds true:
 \begin{equation}\label{Eq_1}
\lim_{t \nearrow T_\ast}w(\alpha,t) = 0
\end{equation}
if and only if 
\begin{equation}\label{Eq_2}
\liminf_{t \nearrow T_\ast } u_x\left(x(\alpha,t),t \right) = -\infty.
\end{equation}

\item If one of \eqref{Eq_1}--\eqref{Eq_2} holds for some $\tilde{\alpha}\in\mathbb{R}$, then  there are uniform constants $c_0,c_1>0$ such that 
\begin{equation}\label{Eq_3}
\frac{c_0}{t-T_\ast} <  u_x\left(x(\tilde{\alpha},t),t \right) < \frac{c_1}{t-T_\ast} 
\end{equation}
for all $t<T_\ast$ sufficiently close to $T_\ast$. 
\end{enumerate}
\end{lemma}

\begin{remark}
\begin{enumerate}
\item By integrating \eqref{var_ODE}, we obtain
\begin{equation}\label{var_ODE1}
w(\alpha,t) =\exp \left( \int_0^t u_x(x(\alpha,s),s) \,ds \right).
\end{equation}
While it is easy by \eqref{var_ODE1} to see that \eqref{Eq_1} implies \eqref{Eq_2}, the converse is not obvious since one cannot exclude the possibility that $u_x$ diverge  in some other earlier time, say $T_0<T_\ast$ with an integrable order in $t$, for which we still have $w(\alpha, T_0)>0$.
%
%
%
%
For the proof of the converse and obtaining the blow-up rate \eqref{Eq_3}, Lemma \ref{phi-bd}, the uniform boundedness of $\phi$, will be crucially used. 
%
%
%
%
%
\item From \eqref{Taylor_w} and \eqref{Taylor_w3}, we see that if $\dot{w}(T_\ast)<0$, the vanishing (or blow-up) order of $w$ (or $\rho$) is $(t-T_\ast)$ (or $(t-T_\ast)^{-1}$) and if $\dot{w}(T_\ast)=0$, the vanishing (or blow-up) order of $w$ (or $\rho$) is $(t-T_\ast)^2$ (or $(t-T_\ast)^{-2}$).
\end{enumerate}
\end{remark}

%

\begin{proof}[Proof of Lemma \ref{Lem_Blowup}]
We suppress the parameter $\alpha$ for notational simplicity. We first make a few basic observations. By the assumption, we have that $w(t)>0$ for all $t\in[0,T_\ast)$. From \eqref{2ndOrdODE} and the fact that $e^\phi w(t) >0$, we obtain that
\begin{equation}\label{2ndOrdODE2}
\ddot{w}(t) < \ddot{w}(t) + e^{\phi(x(\alpha, t),t)} w(t) = \rho_0,
\end{equation}
for which we integrate \eqref{2ndOrdODE2} in $t$ twice to deduce that  $w(t)$ is bounded above on $[0,T_\ast)$.
This together with \eqref{2ndOrdODE} and  Lemma \ref{phi-bd} implies that  $|\ddot{w}(t)|$ is bounded on the interval $[0,T_\ast)$. 
Using this for
\[
\dot{w}(t) -\dot{w}(s) = \int_s^t \ddot{w}(\tau)\,d\tau,
\]
we see that $\dot{w}(t)$ is uniformly continuous on $[0,T_\ast)$. 
Hence, 
we see that the following limit 
\[
 \dot{w}(T_\ast):=\lim_{t \nearrow T_\ast} \dot{w}(t)\in(-\infty,+\infty)
\]
exists.  In a similar fashion, one can check that
\[
 w(T_\ast):=\lim_{t \nearrow T_\ast}w(t)\in[0,+\infty).
\]
We prove the first statement. It is obvious from \eqref{var_ODE1} that \eqref{Eq_1} implies \eqref{Eq_2}. To show that \eqref{Eq_2} implies \eqref{Eq_1}, we suppose $\lim_{t \nearrow T_\ast} w(t) >0$. Then, since $w(0)=1$, $w(t)$ has a strictly positive lower bound on $[0,T_\ast)$. From \eqref{Eq_2}, we may choose a sequence $t_k$ such that $u_x(t_k) \to -\infty$ as $t_k \nearrow T_\ast$. Now using \eqref{var_ODE}, we obtain that
\[
u_x(t_k)w(t_k) - u_x(s)w(s) = \dot{w}(t_k)-\dot{w}(s)  =  \int_s^{t_k} \ddot{w}(\tau)\,d\tau,
\]
which leads a contradiction by letting $t_k \nearrow T_\ast$. Hence, \eqref{Eq_1} holds.


Now we prove the second statement. Due to the first statement, it is enough to assume that \eqref{Eq_1} holds for some $\tilde{\alpha}\in\mathbb{R}$. From  \eqref{2ndOrdODE} and Lemma \ref{phi-bd}, we see that  \eqref{Eq_1} implies
\begin{equation}\label{Taylor_w1}
\lim_{t \nearrow T_\ast} \ddot{w}(t) = \rho_0 >0.
\end{equation} 
 Since  $w(t)>0$ on $[0,T_\ast)$, \eqref{Eq_1} also implies that 
\begin{equation}\label{Taylor_w2}
\dot{w}(T_\ast)= \lim_{t \nearrow T_\ast}\dot{w}(t) \leq 0.
\end{equation}

By the fundamental theorem of calculus, one has $\dot{w}(t) = \dot{w}(\tau) + \int_{\tau} ^t \ddot{w}(s)\,ds$ for all $t, \tau \in[0,T_\ast)$. Then taking the limit $\tau \nearrow T_\ast$ and integrating once more, we obtain that for $t<T_\ast$,
\begin{equation}\label{Taylor_w}
\left.
\begin{array}{l l}
\dot{w}(t) = \dot{w}(T_\ast) + \displaystyle{ \int_{T_\ast} ^t \ddot{w}(s)\,ds, } \\ 
w(t) =  \dot{w}(T_\ast) (t - T_\ast) + \displaystyle{ \int_{T_\ast}^t \ddot{w}(s)(t-s)\,ds. }
\end{array} 
\right.
\end{equation}
Using \eqref{Taylor_w1}, we have that for all $t<T_\ast$ sufficiently close to $T_\ast$,
\begin{equation}\label{Taylor_w3}
\left.
\begin{array}{l l}
\displaystyle{ 2\rho_0 ( t-T_\ast) <  \int_{T_\ast} ^t \ddot{w}(s)\,ds < \frac{\rho_0}{2} ( t-T_\ast), } \\ 
\displaystyle{ \frac{\rho_0}{4}(T_\ast-t)^2 < \int_{T_\ast}^t \ddot{w}(s)(t-s)\,ds < \rho_0(T_\ast-t)^2. }
\end{array} 
\right. 
\end{equation}
 Thanks to \eqref{Taylor_w2}, we note that either $\dot{w}(T_\ast)<0$ or $\dot{w}(T_\ast)=0$ holds. Combining \eqref{Taylor_w}--\eqref{Taylor_w3}, we conclude that  if $\dot{w}(T_\ast)<0$, then
\[
1/2  <(t - T_\ast)u_x = (t - T_\ast)\frac{\dot{w}}{w}  < 2 ,
\]
and if $\dot{w}(T_\ast)=0$, then
\[
1  <(t - T_\ast)u_x = (t - T_\ast)\frac{\dot{w}}{w}  < 8.
\]
This completes the proof of  \eqref{Eq_3}. 
\end{proof}

Now we are ready to prove Theorem~\ref{MainTheorem}.
\begin{proof}[Proof of Theorem~\ref{MainTheorem}]
We consider the equation \eqref{2ndOrdODE} with $\alpha\in \mathbb{R}$, for which \eqref{ThmCon2} holds. Suppose that the smooth solution to \eqref{EP} with $K=0$ exists for all $t\in[0,+\infty)$. Then, thanks to Lemma~\ref{Lem_Blowup}, we must have 
\begin{equation}\label{Assume1}
w(\alpha,t)>0 \quad \text{ for all } t\in[0,+\infty). 
\end{equation}
Combining \eqref{2ndOrdODE} and Lemma \ref{phi-bd}, we have that  for all $t\in [0,+\infty)$,
\begin{equation}\label{DiffIneq}
\ddot{w}(t) + a w(t) \leq b, \quad w(0)\geq 1,
\end{equation} 
where we let
\[
w(t)=w(\alpha,t), \quad a:=\exp\left( V_-^{-1}(H(0)) \right), \quad  b:=\rho_0(\alpha)
\]
for notational simplicity. We notice that the inequality $w(0) \geq 1$ is allowed in \eqref{DiffIneq}. In what follows, we  show  that  there exists a finite time $T_*>0$ such that $\lim_{t \nearrow T_*}w(\alpha,t) =0$. This contradicts to \eqref{Assume1}, and hence finishes the proof of Theorem \ref{MainTheorem}.

We consider two disjoint cases, call them \textit{Case A} and \textit{Case B} for $\dot{w}(0)\leq 0$ and $\dot{w}(0)>0$, respectively.

 
\textit{Case A}: We first consider the case $\dot{w}(0)\leq 0$.  We claim that $b - a w(t) =0$ for some $t$.  Suppose to the contrary that $b -a w(t) \ne 0$ for all $t \geq 0$. 
Since  $b-aw(0)<0$ from \eqref{ThmCon2} and $w(0)\ge1$, we have 
\begin{equation}\label{DiffIneq1}
b -a w(t) < 0 \quad \text{ for all } t \in [ 0,+\infty).
\end{equation} 
Combining \eqref{DiffIneq}--\eqref{DiffIneq1}, we see that $\ddot{w}(t)<0$ for all $t$. From this and $\dot{w}(0)\leq 0$, we have that $\dot{w}(t) \to c \in [-\infty,0)$  as $t \to +\infty$, which implies that  $w(t) \to -\infty$ as $t \to +\infty$. This is a  contradiction to \eqref{DiffIneq1}. This proves the claim. 


Then, by the continuity of $w$, we can choose the minimal $T_1>0$ such that 
\begin{equation}\label{Ineq4}
b=a w(T_1).
\end{equation}
 Hence it holds $\ddot{w}(t) \le b - a w(t)<0$ for all $t \in (0,T_1)$, which in turn implies   
\begin{equation*}
\dot{w}(t)=\int_0^t \ddot{w}(s)\,ds + \dot{w}(0) < 0 \quad \text{for all } t \in (0,T_1].
\end{equation*}

Now we split the proof further into two cases: 
\begin{subequations}
\begin{align}
&  \text{(i)} \quad    \dot{w}(t)<0 \text{ on }  (0,T_1] \text{ and } \dot{w}(t) \text{ has a zero on } (T_1, +\infty), \\
&   \text{(ii)} \quad  \dot{w}(t)<0 \text{ for all } t > 0. \label{Ineq12}
\end{align}
\end{subequations}
%

\textit{Case} (i) :  We choose the minimal $T_2>T_1$ satisfying 
\begin{equation}\label{wpT2}
\dot{w}(T_2)=0.
\end{equation}
Then, $\dot{w}(t) < 0 $ for $t\in (0,T_2)$. It suffices to show that $w(T_2)\le0$ since this implies that 
$w(t)=0$ for some $t\in (0,T_2]$ as desired.

We shall show that $w(T_2) \le 0$ by contradiction.  Suppose not, i.e., $w(T_2)>0$. Then since $w$ decreases on $[T_1,T_2]$, we have 
\begin{equation}\label{Ineq3}
0 < w(T_2) < w(T_1)= b/a,
\end{equation} 
where the equality is from \eqref{Ineq4}. Multiplying \eqref{DiffIneq} by $\dot{w} \leq 0$, and then integrating over $[0,t]$, we obtain that for $t\in[0,T_2],$
\begin{equation}\label{Ineq1}
\frac{|\dot{w}(t)|^2}{2}  \geq -a\left(\frac{w(t)^2-|w(0)|^2}{2} \right) + b(w(t)-w(0)) + \frac{|\dot{w}(0)|^2}{2}.
\end{equation}
Here we define a function $\tilde{g}(w) := -a\left(\frac{w^2-|w(0)|^2}{2} \right) + b(w-w(0)) + \frac{|\dot{w}(0)|^2}{2}$.
We see that 
\begin{equation}\label{Ineq2-0}
\begin{split}
\tilde{g}(0)   = \frac{a|w(0)|^2}{2} -w(0)b+ \frac{|\dot{w}(0)|^2}{2}
  \ge \frac{a}{2} - b+ \frac{|\dot{w}(0)|^2}{2} >0,
\end{split}
\end{equation}
where we have used the assumption $w(0)\ge 1$ and \eqref{ThmCon2} for the last two inequalities, respectively. 
By inspection, one can check that the function $\tilde{g}(w)$ is strictly increasing on $[0,b/a]$. Using this together with \eqref{Ineq2-0}, we have
\begin{equation}\label{Ineq2}
\begin{split}
\tilde{g}(w) 
 \geq \tilde{g}(0) >0 \ \ \text{ for all } w \in[0,b/a].
\end{split}
\end{equation}
Combining \eqref{wpT2}--\eqref{Ineq2}, we have
\[
0 = \frac{|\dot{w}(T_2)|^2}{2} \geq \tilde{g}(w(T_2))  > 0,
\]
which is a contradiction. 


\textit{Case} (ii) : 
%
%
We first claim that $\limsup_{t \to \infty} \dot{w}(t) = 0$. 
If not, i.e., $\limsup_{t \to \infty} \dot{w}(t) \ne 0$, then thanks to \eqref{Ineq12}, 
 we have $\limsup_{t \to \infty} \dot{w}(t)< 0$. This implies $w(t)=0$ for some $t>0$, which is a contradiction to \eqref{Assume1}.

On the other hand, since $w$ is monotonically decreasing on $(0,\infty)$ thanks to \eqref{Ineq12},
we see that $w_\infty:=\lim_{t\to \infty}w(t)$ exists and  $w_\infty  \in [0,b/a]$ by \eqref{Ineq4}.
Similarly as in obtaining \eqref{Ineq1}, we multiply \eqref{DiffIneq} by $\dot{w}(t) \leq 0$, $t\in[0,\infty)$, and then integrate the resultant over $[0,t]$ to obtain that   \eqref{Ineq1} holds for $t\in[0,\infty)$.  Since $0=\limsup_{t \to \infty} \dot{w}(t)=\liminf_{t \to \infty} |\dot{w}(t)|$,   we arrive at 
\[
0 = \liminf_{t \to \infty} |\dot{w}(t)|^2/2 \geq \liminf_{t \to \infty} \tilde{g}(w(t)) = \tilde{g}(w_\infty)\geq \tilde{g}(0)  >  0,
\]
where we have used \eqref{Ineq2} for the last inequality. 
This is absurd, which completes the proof for \textit{Case A}. 

\textit{Case B}: Now we consider the case $\dot{w}(0)> 0$. We claim that  $\dot{w}(t)=0$ for some $t>0$. If not, i.e., $\dot{w}(t)>0$ for all $t \geq 0$, we have 
\[
\ddot{w}(t) \leq b - a w(t) \leq b - aw(0) < 0.
\]
This implies that  $\dot{w}(t) \to -\infty$ as $t \to +\infty$, which is a contradiction to the assumption that $\dot{w}(t)>0$ for all $t\ge0$. 

By the continuity of $\dot{w}(t)$, there is a minimal number $T_0>0$ such that $\dot{w}(T_0) = 0$. Since $\dot{w}(t)>0$ for $t\in[0,T_0)$, we see that $w(T_0)\geq w(0) \geq 1$. 
Now one can apply the same argument as \textit{Case A}  to conclude that $w(t)$ has a zero on the interval $[T_0,+\infty)$.
This completes the proof of Theorem \ref{MainTheorem}.
\end{proof}

We remark that, following the proof of Theorem \ref{MainTheorem}, one obtains an interesting lemma concerning the existence of zeros of second-order linear differential inequality (see Appendix \ref{Appen2}).

\section{Numerical experiments and discussions} \label{numerical}

In this section, we present numerical simulations concerning our blow-up results presented in Theorem~\ref{MainThm_Warm} and Theorem~\ref{MainTheorem}. Based on our numerical observations, we also discuss   quantitative and qualitative differences between the two models, i.e., $K=0$ and $K>0$, in the behaviors of their solutions. Referring to \cite{LS}, the implicit pseudo-spectral scheme with $\Delta x=10/2^{10}$ is employed to solve \eqref{EP} numerically on periodic domains  for numerical convenience.   The Crank–Nicolson method with $\Delta t=0.01$ is applied for time marching.

\begin{table}[h]
\begin{tabular}{c|c|c|c}
&(a)&(b)&(c)\\\hline
$\rho_0(x)$&$1-0.7 \text{sech}(3x)$&$1-0.7 \text{sech}(2x)$&$1-0.3 \text{sech}(2x)$\\\hline
 $H(0)$  & $0.0875$& $0.1671$& $0.0036$\\\hline
$\exp(f_{-}^{-1}(H(0)))$ &$0.6448$&$0.5390$&$0.7585$\\\hline
Blow-up condition (\ref{ThmCon2}) & Hold  & Not hold & Not hold \\\hline
Numerical results & Figure~\ref{f1} & Figure \ref{f2} & Figure \ref{f3} \\
\end{tabular}
\caption{The pressureless case. $\rho_0$ is the initial density function. The initial velocity $u_0$ are given as identically zero function for all cases. $H(0)$ is the energy defined in \eqref{EnergyConser} for the initial data.}\label{Table1}
\end{table}

\begin{figure}[tbhp!]
{\includegraphics[width=140mm,height=90mm]{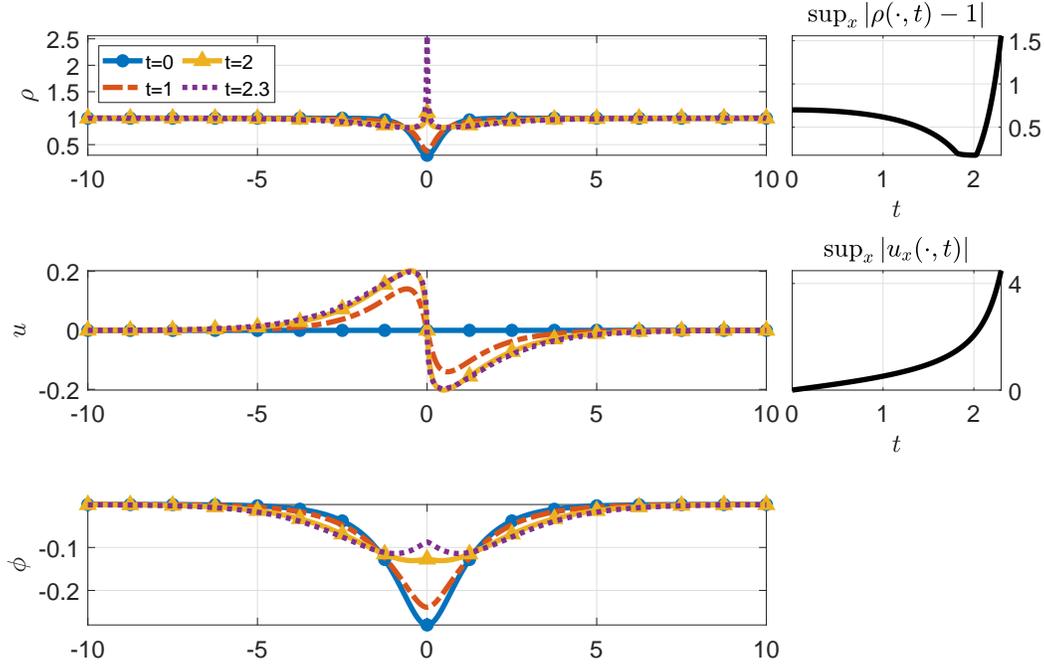}}
\caption{Numerical solution to the pressureless Euler-Poisson system for the case (a) in Table \ref{Table1}. The initial data are $\rho_0 = 1 - 0.7\text{sech}(3x)$ and $u_0 \equiv 0$. $\rho(0,t)$ and $-u_x(0,t)$ are getting larger as $t$ increases, and they blow up in a finite time.} \label{f1}
\end{figure}

We first numerically solve the pressureless model, i.e., \eqref{EP} with $K=0$, for which we consider three cases (see Table \ref{Table1}). In case (a), where  condition (\ref{ThmCon2}) holds,  we observe that $\rho$ and $u_x$   blow up after $t=2.3$ in Figure \ref{f1}. This supports our result in Theorem~\ref{MainTheorem}.  In case (b),  we find that the solutions are bound to break down after $t=2.7$ in Figure \ref{f2} while  condition (\ref{ThmCon2}) is not satisfied. This indicates that the blow-up condition (\ref{ThmCon2}) has a room to be improved.   Lastly, in case (c), where condition \eqref{ThmCon2} is not satisfied,  the smooth solutions seem to persist for $t\in[0,20]$ in Figure \ref{f3}.

\begin{figure}[h]
 \includegraphics[width=140mm,height=60mm]{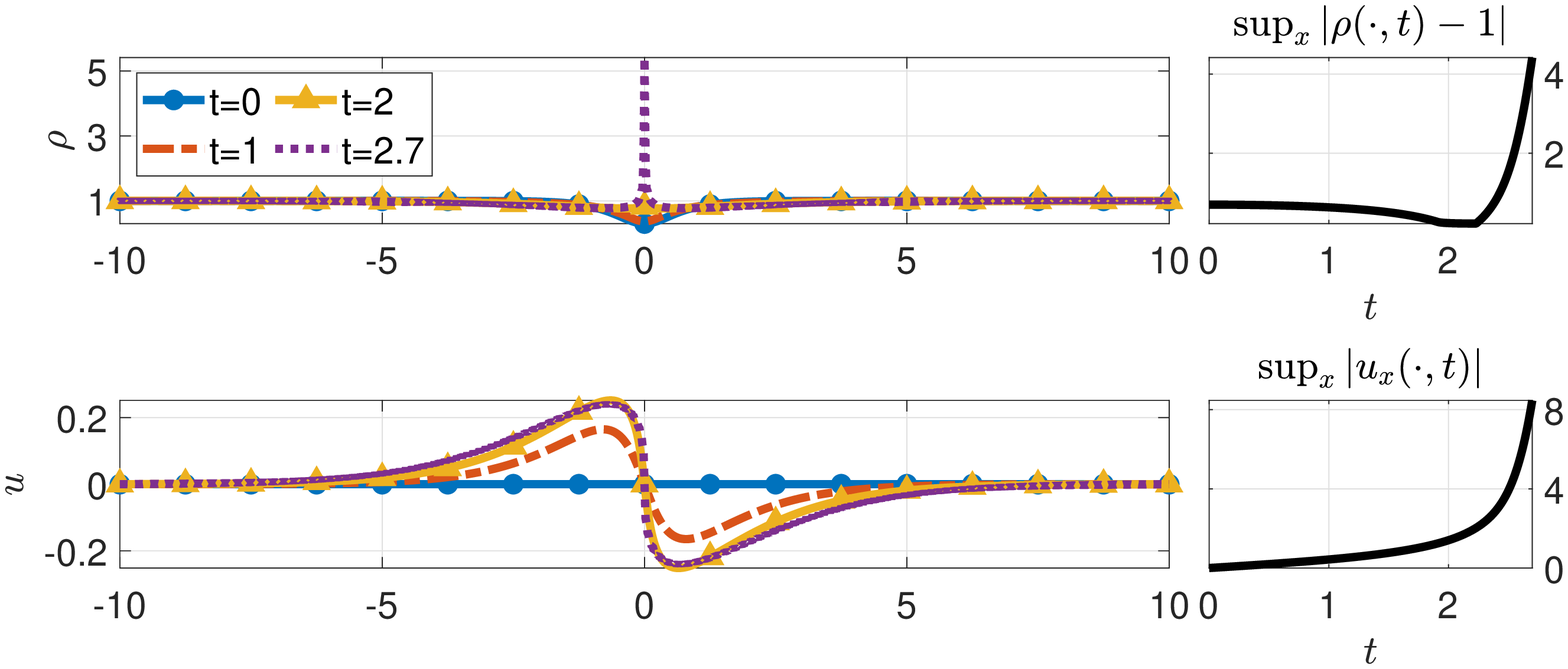}
\caption{Numerical solution to the pressureless Euler-Poisson system for the case (b) in Table \ref{Table1}. The initial data are $\rho_0 = 1 - 0.7\text{sech}(2x)$ and $u_0 \equiv 0$. Although the condition \eqref{ThmCon2} does not hold,  $\rho(0,t)$ and $-u_x(0,t)$ are expected to eventually blow up at a finite time. } \label{f2}
\end{figure}

\begin{figure}[h]
 \includegraphics[width=140mm,height=60mm]{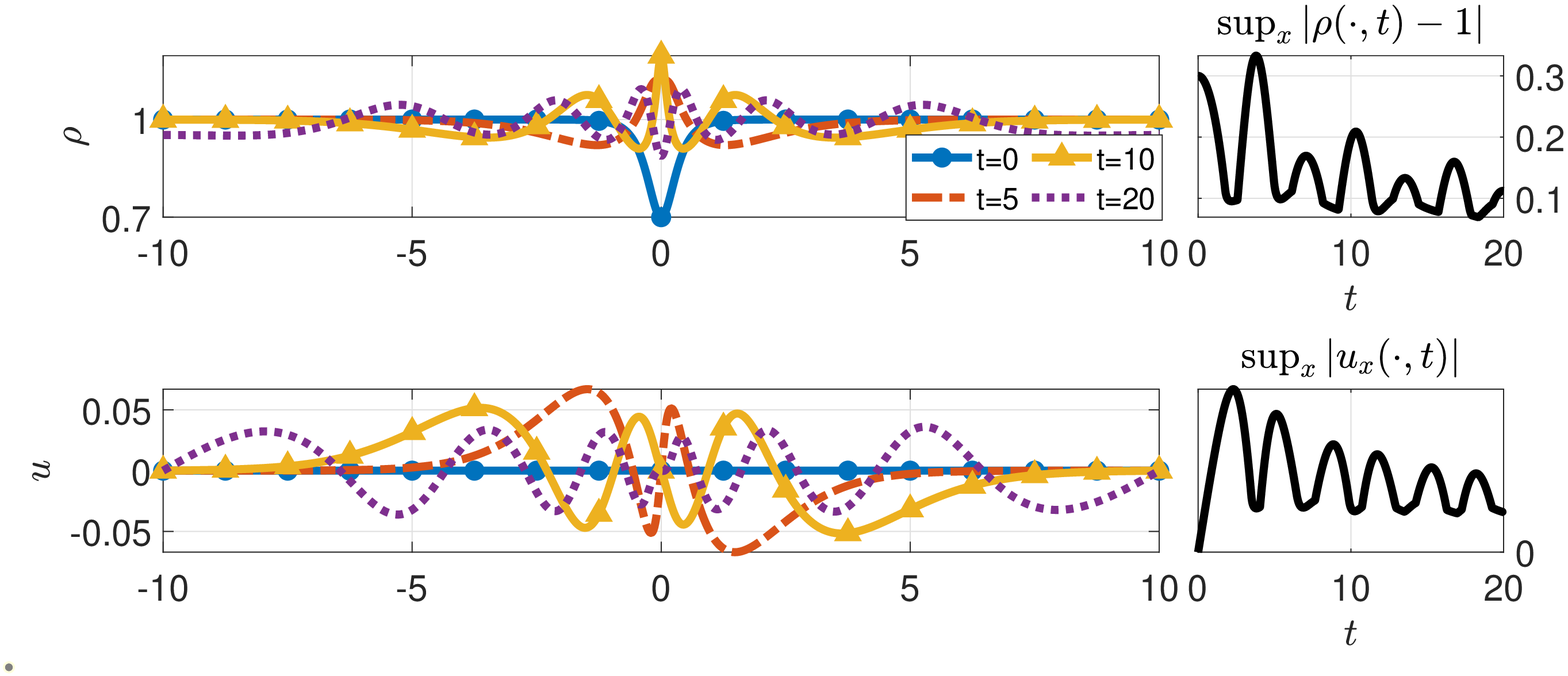}
\caption{Numerical solution to the pressureless Euler-Poisson system for the case (c) in Table \ref{Table1} when the initial conditions does not hold \eqref{ThmCon2}.  The solutions keep oscillating and decreasing as times t goes by. } \label{f3}
\end{figure}

Now we consider the isothermal model, i.e., \eqref{EP} with $K>0$. Figure \ref{Fig3} shows the numerical solution to the isothermal ($K=0.5$) Euler-Poisson system. The same initial data are taken as those for case (a) of the pressureless model (Figure \ref{f1}). We observe that the solution blows up in a short time, but it occurs in a different way from the pressureless case; $\rho$ and $u$ stay bounded while their gradients blow up. In fact, this feature is asserted and proved in Theorem~\ref{MainTheorem}. We remark that unlike the pressureless case, the gradient blow-up occur near the origin, not at the origin. This makes sense since the models have different nature in terms of characteristic curves. 
In fact, the ion waves propagate very differently as illustrated in Figure \ref{Fig4}. In the pressureless case, $\rho$ oscillates at $x=0$ and the resulting oscillatory waves propagate. In the isothermal case, in contrast, such an oscillatory behaviors does not appear around $x=0$. Instead, some localized waves propagate with some oscillatory waves. This difference  is caused by their different ``characteristic curves". Hence, the following question is naturally posed: does the blow-up occur due to the collision of the dominant characteristics associated with $u\pm \sqrt{K}$? For both models, the analytical study on the mechanism of ion-wave propagation and blow-up would be very interesting and challenging due to the coupled electric potential term, which gives rise to a  dispersive effect.

\begin{figure}[tbhp!] 
{ \includegraphics[width=140mm,height=60mm]{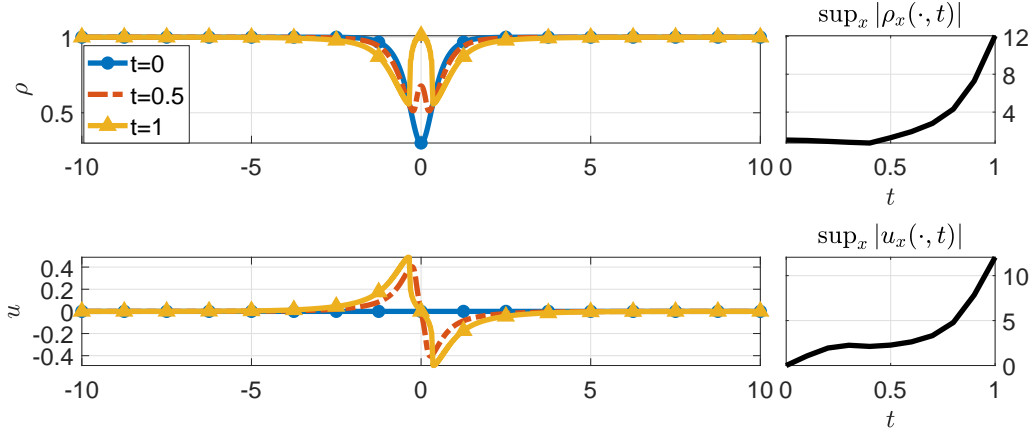}}
  \caption{Numerical solution to the isothermal ($K=0.5$) Euler-Poisson system.  The initial data are $\rho_0 = 1 - 0.7\text{sech}(3x)$ and $u_0 \equiv 0$. $\|\partial_x(\rho,u)(\cdot,t)\|_{L^\infty}$ blows up in a finite time while $\|(\rho,u)(\cdot,t)\|_{L^\infty}$ is bounded.} \label{Fig3}
\end{figure}

\begin{figure}[tbhp!]
 \includegraphics[width=140mm,height=30mm]{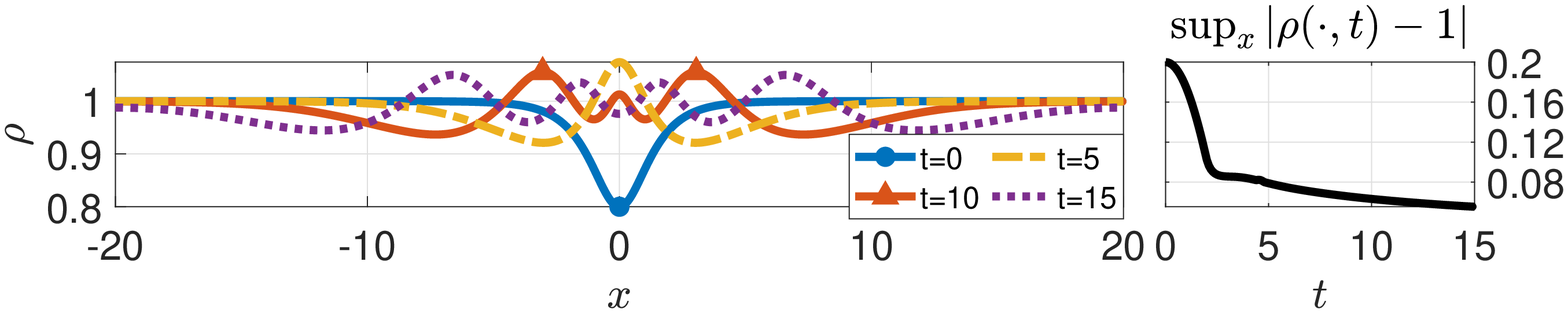}\\ 
 \includegraphics[width=140mm,height=30mm]{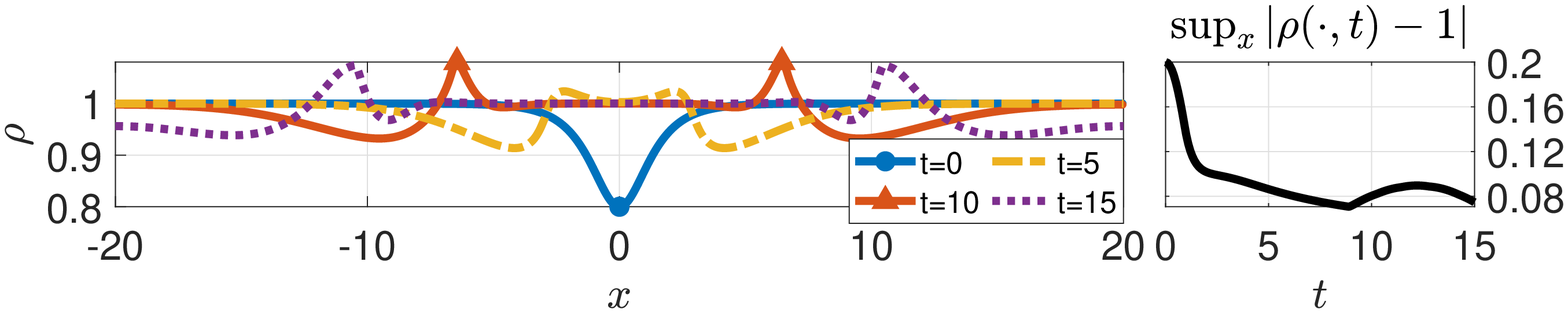} \\
      \caption{Numerical solution to the pressureless (above) and the isothermal (below, $K=0.5$) Euler-Poisson systems. The initial data are taken as $\rho_0 = 1-(0.2)\text{sech}(x)$ and $u_0 \equiv 0$.  } \label{Fig4}
\end{figure}

\begin{figure}[tbhp!] 
 \includegraphics[width=140mm,height=30mm]{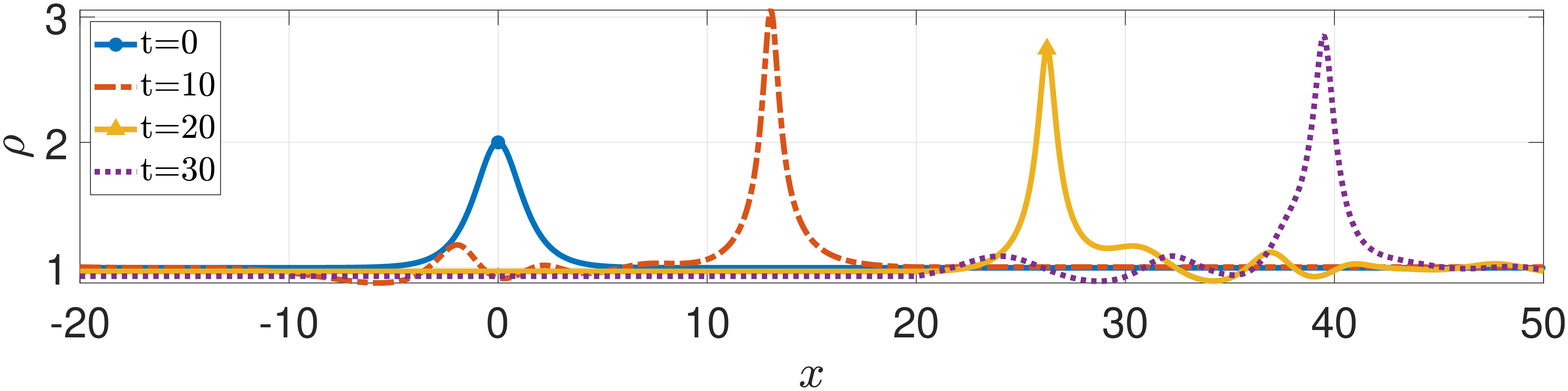} \\ 
  (a) $K=0$ \\
 \includegraphics[width=140mm,height=30mm]{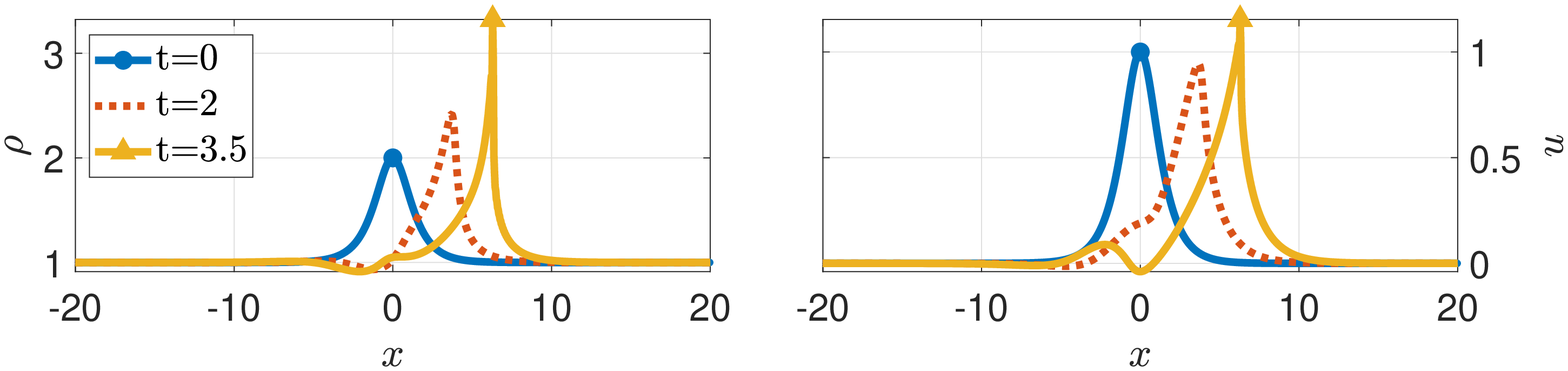} \\
 (b) $K=0.5$ \caption{Numerical solutions to the pressureless (above) and the isothermal (below) Euler-Poisson systems. The initial data are taken as  $\rho_0=1+\text{sech}(x)$ and $u_0=\text{sech}(x)$. The numerical solution of the pressureless case persists for a long time while that of the isothermal case blows up in a short time.} \label{Fig2}
\end{figure} 

\begin{table}[h]
\begin{tabular}{c|c|c|c}
&&$K=0$&$K=0.5$\\\hline\hline
\multirow{4}{*}{Comparison 1}&$\rho_0(x)$&\multicolumn{2}{c}{$1-0.7 \text{sech}(3x)$}\\\cline{2-4}
&$u_0(x)$&\multicolumn{2}{c}{$0$}\\\cline{2-4}
&Blow-up & O  & O \\\cline{2-4}
&Numerical results & Figure~\ref{f1} & Figure \ref{Fig3} \\\hline\hline
\multirow{4}{*}{Comparison 2}&$\rho_0(x)$&\multicolumn{2}{c}{$1-0.2 \text{sech}(x)$}\\\cline{2-4}
&$u_0(x)$&\multicolumn{2}{c}{$0$}\\\cline{2-4}
&Blow-up & X  & X \\\cline{2-4}
&Numerical results & Figure~\ref{Fig4}(a) & Figure \ref{Fig4}(b) \\\hline\hline
\multirow{4}{*}{Comparison 3}&$\rho_0(x)$&\multicolumn{2}{c}{$1+\text{sech}(x)$}\\\cline{2-4}
&$u_0(x)$&\multicolumn{2}{c}{$\text{sech}(x)$}\\\cline{2-4}
&Blow-up & X  & O \\\cline{2-4}
&Numerical results & Figure~\ref{Fig2}(a) & Figure \ref{Fig2}(b)
\end{tabular}
\caption{Comparisons between the pressureless model  and the isothermal model.}\label{Table2}
\end{table}

We shall discuss how $\partial_x\rho_0$ affect on the blow-up. In Figures \ref{Fig2}, we present numerical   solutions to the Euler-Poisson system both for the pressureless and the isothermal cases with the same initial data. For the pressureless case (Figure \ref{Fig2}.(a)), the initially localized (compressed) wave travels to the right, and it persists for a long time while the numerical solution of the isothermal case (Figure \ref{Fig2}.(b)) blows up in a short time.  
We conjecture for the pressureless case that $\partial_x\rho_0$ itself is not a critical component which makes the finite time blow-up occur. 
It becomes more plausible by the fact that for any given constant $M>0$, there is a smooth traveling solitary wave solution satisfying $\sup_{x\in \mathbb{R}}|\partial_x\rho(\cdot,t)|>M$ for all $t \geq 0$. In fact, as the speed $c\nearrow c_0$, where $c_0>1$ is some critical speed, the maximum value of the traveling solitary wave $\rho_c$ tends to infinity. (See Section 8.2 of \cite{BK} and Figure 6 of \cite{BK}). For the isothermal case, however, there is an upper bound (depending only on $K$) for $\sup_{c}\|\partial_x \rho_c(\cdot,t)\|_{L^\infty}$. Based on our numerical experiments, we find that the larger $K$ is, we get the smaller this upper bound, see Figure 6 of \cite{BK}.

We also remark that 
the numerical solution in Figure~\ref{Fig4}  seems to converge to a background constant state $(\rho, u, \phi)=(1,0,0)$. In the study of long time dynamics, a question of whether smooth solutions globally exist or  blow up in finite time arises naturally.
As mentioned in the introduction, the question if the global smooth solutions scatter to the constant state is conjectured by the dispersion relation of the linearized Euler-Poisson system.
%
The questions of global existence of smooth solutions and their long time behavior are intriguing and challenging.

\section{Appendix}

\subsection{Zeros of second-order differential inequality}\label{Appen2}
Following the proof of Theorem \ref{MainTheorem}, one obtains the following  lemma:
\begin{lemma}\label{Lem_DiffIneq}
Let $a$ and $b$ be positive constants. Suppose $w(t)$ satisfies 
\[
\ddot{w} + a w \leq b
\]
for all $t \geq T_0$ and $w(T_0)\geq 1$. If $a/2>b$ and 
\begin{equation}\label{Con11}
\frac{a|w(T_0)|^2}{2} -w(T_0)b+ \frac{|\dot{w}(T_0)|^2}{2} > 0,
\end{equation}
then $w(t)$ has a zero on the interval $(T_0,+\infty)$. 
\end{lemma}

The authors are not aware of any literature addressing the existence of zeros of second-order linear  differential inequality with the coefficient $a>0$ and  \textit{constant nonhomogeneous} term $b$.  We finish this subsection with some remarks regarding Lemma \ref{Lem_DiffIneq}.

\begin{remark}
\begin{enumerate}
\item For the case of the differential equation  $\ddot{w}+aw = b$, 
\begin{equation}\label{Con1}
\frac{a|w(0)|^2}{2} -w(0)b+ \frac{|\dot{w}(0)|^2}{2} \geq 0
\end{equation}
 is a necessary and sufficient condition in order for $w$ to have a zero on $[0,+\infty)$. 
 \item One needs the restriction $a/2>b$ (or $a/2\geq b$)  in Lemma \ref{Lem_DiffIneq}. If $a/2<b$, then the solution to  $\ddot{w} + aw = b$ with $w(0)=1$ and $\dot{w}(0)=0$ has no zero since \eqref{Con1} is not satisfied. For another example, we consider the equation
\begin{equation}\label{DiffIneq0}
\ddot{w} + aw = b - e^{-t}, \quad t \in [0,+\infty),
\end{equation}
where $a,b>0$ are constants. Since the general solution of \eqref{DiffIneq0} is 
\[
w(t) = \alpha \cos \sqrt{a} t + \beta \sin \sqrt{a} t + \frac{b}{a} - \frac{e^{-t}}{a+1},
\]
we have  
\[
w(0)=\alpha + \frac{b}{a} - \frac{1}{a+1}, \quad \dot{w}(0) = \sqrt{a}\beta + \frac{1}{a+1}.
\]
Since 
\[
\begin{split}
\min_{t \geq 0} w(t)
&  \geq \min_{t \geq 0} \left(\alpha \cos \sqrt{a} t + \beta \sin \sqrt{a} t \right) + \min_{t \geq 0} \left( \frac{b}{a} - \frac{e^{-t}}{a+1} \right) \\
&  = -\sqrt{\alpha^2+\beta^2} + \frac{b}{a} - \frac{1}{a+1},
\end{split}
\]
$w(t)$ has no zero on $[0,+\infty)$ provided that
\begin{equation}\label{Aux00}
-\sqrt{\alpha^2+\beta^2} + \frac{b}{a} - \frac{1}{a+1} > 0.
\end{equation}
We choose $b=1/3$ and $a>0$ sufficiently small such that
\begin{equation}\label{Aux0}
\frac{1}{2(a+1)^2} > b-\frac{a}{2} > \frac{a}{a+1}.
\end{equation}
For $w(0)=1$ and $\dot{w}(0)=\frac{1}{a+1}$, the first inequality of \eqref{Aux0} is equivalent to \eqref{Con11} and the second inequality of \eqref{Aux0} is equivalent to \eqref{Aux00}. On the other hand, $b>a/2$ holds.
\end{enumerate}
\end{remark} 

\subsection{Proof of inequality \eqref{EnergyBd}}\label{Appen1}
Lemma \ref{LemmaAppen} is derived from some elliptic estimates (see \cite{LLS}). The inequality \eqref{EnergyBd} follows from \eqref{PhiH1Bd2} and the definition of $H(t)$.
\begin{lemma}
\label{LemmaAppen}
For $\rho-1 \in L^\infty(\mathbb{R}) \cap L^2(\mathbb{R})$ satisfying $\inf_{x \in \mathbb{R}}\rho>0$ and $\lim_{|x|\to  \infty } \rho = 1$, let $\phi$ be the solution to the Poisson equation \eqref{EP_3}. Then, the following hold:
\begin{enumerate}
\item 
\begin{equation}\label{MaxPrincp}
 \kappa_-:=\inf_{x\in\mathbb{R}}\rho \leq e^\phi \leq \sup_{x\in\mathbb{R}}\rho  \quad \text{for all } x \in \mathbb{R},
\end{equation}
\item if $1> \kappa_- > 0$, then 
\begin{subequations}
\begin{align}
&  \int_{\mathbb{R}} |\phi_x|^2 + \frac{\kappa_0}{2}|\phi|^2\,dx \leq \frac{1}{2 \kappa_0}\int_{\mathbb{R}} |\rho-1|^2\,dx, \quad   \label{PhiH1Bd} \\
& \int_{\mathbb{R}} |\phi_x|^2 + (\phi-1)e^\phi +1  \,dx \leq \frac{1}{\kappa_0} \int_{\mathbb{R}} |\rho-1|^2\,dx, \label{PhiH1Bd2}
\end{align}
\end{subequations}
where $\kappa_0:= \frac{1-\kappa_-}{-\log \kappa_-}$.
\item if $\kappa_- \geq 1$, then \eqref{PhiH1Bd} and \eqref{PhiH1Bd2} hold with $\kappa_0 =1$.
\end{enumerate} 
\end{lemma}
\begin{proof}
The maximum principle \eqref{MaxPrincp} can be proved by Stampacchia's truncation method. Referring to \cite{LLS}, we omit the details. Multiplying the Poisson equation \eqref{EP_3} by $\phi$ and  integrating by parts, we have
\begin{equation}\label{AppendE3}
\int (\rho-1)\phi \,dx = \int |\phi_x|^2 + (e^\phi-1)\phi \,dx.
\end{equation}


We prove the second statement. We first prove \eqref{PhiH1Bd}. Letting $\kappa:= - \log \kappa_->0$,
we claim that 
\begin{equation}\label{AppendEq2}
(e^\phi-1)\phi \geq \frac{1-e^{-\kappa}}{\kappa}\phi^2 \ \ \text{ for }  \phi \geq  -\kappa. 
\end{equation}
For $0>\phi \geq -\kappa$, we have 
\begin{equation}\label{AppendEq1}
\frac{1-e^\phi}{-\phi} \geq \frac{1-e^{-\kappa}}{\kappa} > 0
\end{equation}
since the mapping  $x \mapsto \frac{1-e^{-x}}{x}$ strictly decreases on $x>0$. Multiplying \eqref{AppendEq1} by $\phi^2$, we get \eqref{AppendEq2} for $0>\phi \geq -\kappa$. On the other hand, since $e^\phi - 1 \geq \phi$ and $1>\frac{1-e^{-x}}{x}$ for $x>0$, we obtain \eqref{AppendEq2} for $\phi \geq 0$. This proves the
 claim. Then \eqref{PhiH1Bd} follows from \eqref{AppendE3} and \eqref{AppendEq2} by applying Young's inequality.

Next we prove \eqref{PhiH1Bd2}. From \eqref{AppendE3} and the fact that $e^\phi - 1 -\phi \geq 0$, we have 
\begin{equation}\label{Equ1}
\begin{split}
\int (\rho-1)\phi \,dx
& = \int |\phi_x|^2 + (\phi-1)e^\phi +1 + (e^\phi-1-\phi) \,dx \\
& \geq \int |\phi_x|^2 + (\phi-1)e^\phi +1  \,dx.
\end{split}
\end{equation}
 Using Young's inequality, we obtain from \eqref{PhiH1Bd} that
\begin{equation}\label{Equ2}
\begin{split}
\int (\rho-1)\phi\,dx 
& \leq \frac{\kappa_0}{2}\int |\phi|^2 \,dx + \frac{1}{2\kappa_0}\int|\rho-1|^2\,dx \\
& \leq  \frac{1}{\kappa_0} \int |\rho-1|^2\,dx.
\end{split}
\end{equation}
Now \eqref{PhiH1Bd2} follows from \eqref{Equ1} and \eqref{Equ2}.
The last statement can be easily checked since $\phi \geq 0$ if $\kappa_- \geq 1$.
\end{proof}

\subsection*{Acknowledgments.}
B.K. was supported by Basic Science Research Program through the National Research Foundation of Korea (NRF) funded by the Ministry of science, ICT and future planning (NRF-2020R1A2C1A01009184). 
The authors thank Shih-Hsin Chen and Yung-Hsiang Huang for suggesting the example \eqref{DiffIneq0}.

\subsection*{Conflict of Interest}
The authors declare that they have no conflict of interest.

\subsection*{Availability of Data}
The data supporting the findings of this study are available within the article.


\begin{thebibliography}{10}
 


\bibitem{BK2} Bae, J., Kwon, B.: Small amplitude limit of solitary waves for the Euler-Poisson system. J. Differential Equations \textbf{266}, 3450-3478 (2019) 

\bibitem{BK} Bae, J., Kwon, B.: Linear stability of solitary waves for the isothermal Euler-Poisson system, 
Arch. Ration. Mech. Anal., 243 (2022) 257-327 

\bibitem{CCTT} Carrillo, J. A., Choi, Y.-P., Tadmor E. and  Tan, C.,: Critical thresholds in 1D Euler equations with non-local forces, Mathematical Models and Methods in Applied Sciences Vol. 26, No. 01, pp. 185-206 (2016)

\bibitem{Ch} Chen, F.F.: Introduction to plasma physics and controlled fusion. 2nd edition, Springer (1984)


\bibitem{CP} Cordier, S., Peng, Y.J.: Syst\'eme Euler-Poisson non lin\'eaire-existence globale de solutions faibles entropiques, Mod. Math. Anal. Num. 32 (1), 1-23, (1998)

\bibitem{Cor} Cordier, S., Degond, P., Markowich, P, Schmeiser, C.: Travelling wave analysis of an isothermal Euler-Poisson model. Ann. Fac. Sci. Toulouse Math. \textbf{5}, 599-643 (1996)

\bibitem{Daf1} Dafermos, C. M.: Dissipation in materials with memory, in Viscoelasticity and Rehology, Academic Press (1984), Orlando, p. 125-156.

\bibitem{Daf2} Dafermos, C. M.: Development of singularities in the motion of materials with fading memory, Arch. Rational Mech. Anal. 91 (1985), 193-205.

\bibitem{DH} Dafermos, C. M. and Hsiao, L.: Development of singularities in solutions of the equations of nonlinear thermoelasticity, Q. Appl. Math. 44 (1986), 463-474

\bibitem{WC} Wang, D. and Chen, G.-Q.: Formation of singularities in compressible Euler--Poisson fluids with heat diffusion and damping relaxation, Journal of Differential Equations, 144 (1998) 44-65.

\bibitem{Dav} Davidson, R.C.: Methods in nonlinear plasma theory. 1st edition, Academic Press (1972)

\bibitem{Dav2} Davidson, R.C. and Schram, P.P.: Nonlinear oscillations in a cold plasma, Nucl. Fusion 8, 183 (1968)


\bibitem{ELT} Engelberg, S., Liu, H., Tadmor, E.:  Critical threshold in Euler-Poisson equations, Indiana Univ. Math. J. 50, 109-157 (2001)

\bibitem{hk} G\'erard-Varet, D., Han-Kwan, D. and  Rousset, F.: Quasineutral limit of the Euler-Poisson system for ions in
a domain with boundaries II, J. Ec. polytech. Math. 1, 343-386 (2014)

\bibitem{GGPS} Grenier, E., Guo, Y., Pausader, B. and Suzuki, M.:  
Derivation of the ion equation,
Quart. Appl. Math. 78, 305-332 (2020)

\bibitem{GP} Guo, Y.; Pausader, B.: Global Smooth Ion Dynamics in the Euler-Poisson System, Commun. Math. Phys. 303: 89 (2011) 


\bibitem{Guo} Guo, Y., Pu, X.: KdV limit of the Euler-Poisson system, Arch. Ration. Mech. Anal. \textbf{211}, 673-710 (2014)

\bibitem{HNS} Haragus, M., Nicholls, D. P., Sattinger, D. H.: Solitary wave interactions of the Euler-Poisson equations, J. Math. Fluid Mech., 5 (2003), pp. 92–118.

\bibitem{HS} Haragus, M., Scheel, A. : Linear stability and instability of ion-acoustic plasma solitary waves. Physica D \textbf{170}, 13-30 (2002)

\bibitem{HJL} Holm, D., Johnson, S.F., Lonngren, K.E.: Expansion of a cold ion cloud, Appl. Phys. Lett. 38, 519 (1981)

\bibitem{Ince} Ince, E. L.: Ordinary Differential Equations, New York: Dover Publications, (1956)

\bibitem{LLS}  Lannes, D., Linares, F. and  Saut, J.C.: The Cauchy problem for the Euler-Poisson system and derivation of the Zakharov-Kuznetsov equation, in: M. Cicognani, F. Colombini, D. Del Santo (Eds.), Studies in Phase Space Analysis with Applications to PDEs, in: Progr. Nonlinear Differential Equations Appl., vol. 84, Birkhäuser, pp. 183-215  (2013)

\bibitem{Lax} Lax, P. D.,: Development of singularities of solutions of nonlinear hyperbolic partial differential equations, J. Math. Physics 5 (1964), 611-613.

\bibitem{LS} Li, Y., Sattinger, D.: Soliton Collisions in the Ion Acoustic Plasma Equations. J. math. fluid mech. 1, 117-130 (1999) 

\bibitem{Liu} Liu, H.: Wave breaking in a class of nonlocal dispersive wave equations. Journal of Nonlinear Mathematical Physics Volume 13, Number 3, 441-466 (2006)


\bibitem{LT1} Liu, H., Tadmor, E.: Spectral dynamics of the velocity gradient field in restricted flows. Commun. Math. Phys. 228(3), 435–466 (2002)

\bibitem{LT} Liu, H., Tadmor, E.: Critical thresholds in 2D restricted Euler–Poisson equations. SIAM J. Appl. Math. 63(6), 1889–1910 (2003)

\bibitem{PuNLS} Liu, H., Pu, X.: Justification of the NLS Approximation for the Euler–Poisson Equation. Commun. Math. Phys. 371, 357–398 (2019) 

\bibitem{suzuki} Nishibata, S., Ohnawa, M. and Suzuki, M.: Asymptotic stability of boundary layers to the Euler-Poisson equations arising in plasma physics, SIAM J. Math. Anal. 44, 761-790 (2012)

\bibitem{PHGOA} Perego, M., Howell, P.D., Gunzburger, M.D., Ockendon, J.R., Allen J. E.: The expansion of a collisionless plasma into a plasma of lower density, Phys. Plasmas 20, 052101 (2013) 

\bibitem{Pecseli} Pécseli, H. L.: Waves and Oscillations in Plasmas. Taylor and Francis, London, (2013)

\bibitem{Pu} Pu, X.: Dispersive Limit of the Euler--Poisson System in Higher Dimensions, SIAM J. Math. Anal., 45(2), 834-878 (2013)

\bibitem{Rie} Riemann, B.: \"Uber die Fortpflanzung ebener Luftwellen von endlicher Schwingungsweite, Abhandlungen der K\"oniglichen Gesellschaft der Wissenschaften in G\"ottingen 8 (1860), 43–66.

\bibitem{Sag} Sagdeev, R.Z.: Cooperative phenomena and shock waves in collisionless plasmas. ``Reviews of Plasma Physics" (M. A. Leontoich, ed.), Vol.IV, Consultants Bureau, New York, 23-91 (1966)




 
\end{thebibliography}
 \end{document}